\documentclass[11pt, oneside]{article}

\usepackage{amsmath,amsfonts,amssymb,amsthm,enumerate,mathrsfs,amscd,bbm}

\usepackage[top=16mm, bottom=16mm, left=44.45mm, right=44.45mm]{geometry}

\usepackage[breaklinks=true,colorlinks=true]{hyperref}
\usepackage[dvipsnames]{xcolor}

\usepackage{titletoc}
\titlecontents{section}[1.5em]{\vspace{-2pt}}%
    {\thecontentslabel\enspace\enspace}%
    {}%
    {\titlerule*[0.5pc]{.}\contentspage}%
\titlecontents{subsection}[3em]{\vspace{-2.5pt}}%
    {\thecontentslabel\enspace\enspace}%
    {}%
    {\titlerule*[0.5pc]{.}\contentspage}%

\numberwithin{equation}{section}
\theoremstyle{plain}
\newtheorem{theorem}{Theorem}[section]

\newtheorem{lemma}[theorem]{Lemma}
\newtheorem{proposition}[theorem]{Proposition}

\theoremstyle{remark}
\newtheorem{example}[theorem]{Example}
\newtheorem{remark}{Remark}[section]

\theoremstyle{definition}
\newtheorem{definition}{Definition}[section]

\newcommand{\bke}[1]{\left ( #1 \right )}

\newcommand{\norm}[1]{\left \| #1 \right \|}
\newcommand{\bka}[1]{{\langle #1 \rangle}}
\newcommand{\abs}[1]{\left | #1 \right |}

\newcommand\al{\alpha}

\newcommand\de{\delta}
\newcommand\e {\varepsilon}

\renewcommand\th{\theta}
\newcommand\ka{\kappa}

\newcommand\De{\Delta}

\newcommand\La{\Lambda}

\newcommand\Om{\Omega}
\newcommand{\R}{\mathbb{R}}
\newcommand{\Z}{\mathbb{Z}}
\newcommand{\G}{\mathbf{G}}

\renewcommand{\div}{\mathop{\mathrm{div}}\nolimits}

\newcommand{\diam}{\mathop{\mathrm{diam}}\nolimits}
\newcommand{\pd}{\partial}
\newcommand{\nb}{\nabla}
\newcommand{\td}{\tilde}

\newcommand{\lec}{{\ \lesssim \ }}

\newcommand{\I}{\infty}
\newcommand{\tsum}{{\textstyle \sum}}
\newcommand{\EQ}[1]{\begin{equation} #1 \end{equation}}

\newcommand{\EQN}[1]{\begin{equation*}\begin{split} #1 \end{split}\end{equation*}}
\newcommand{\bb}{\mathbf{b}}

\newcommand{\eps}{\varepsilon}

\title{On linear elliptic equations
with drift terms in critical weak spaces}
\author{Hyunseok Kim\thanks{%
Department of Mathematics, Sogang University, Seoul, 121-742, Korea (kimh@sogang.ac.kr).}
\and Tuoc Phan\thanks{%
Department of Mathematics, University of Tennessee - Knoxville, Knoxville TN 37996, USA (phan@utk.edu).}
\and Tai-Peng Tsai\thanks{%
Department of Mathematics, University of British Columbia, Vancouver BC V6T1Z2, Canada (ttsai@math.ubc.ca).} %
}
\date{}

\begin{document}
\baselineskip 16pt

\maketitle

\begin{abstract}

We study the Dirichlet problem for a second order linear elliptic equation in a bounded smooth domain $\Omega$ in $\R^n$, $n \ge 3$, with the drift  $\bb $ belonging to  the critical weak space $L^{n,\infty}(\Omega )$. We decompose the drift $\bb = \bb_1 + \bb_2$ in which $\div \bb_1 \geq 0$ and $\bb_2$ is small only in a small scale quasi-norm  of $L^{n,\infty}(\Omega )$.
Under this new smallness condition, we  prove existence, uniqueness, and regularity estimates of weak solutions to the problem and its dual.  H\"{o}lder regularity and derivative estimates of weak solutions to the dual problem are also established. As a result,
we prove uniqueness of very weak solutions slightly below the threshold.  When $\bb_2 =0$, our results recover those by Kim and Tsai in [SIAM J. Math. Anal. 52 (2020)]. Due to the new small scale quasi-norm, our results are new even when $\bb_1=0$.


\end{abstract}

\setcounter{tocdepth}{2}
\tableofcontents

\section{Introduction} \label{Sec1}

Let $\Om$ be a bounded domain in $\R^n$, where $n \ge 3$. In this paper, we consider  the following Dirichlet problem and its dual for linear elliptic equations of second order in divergence form:
\begin{equation}
\label{bvp}
\left\{\begin{array}{rr}
 -\De u + \div( u \mathbf{b} ) +cu = f \quad \text{in }\Om ,\,\,\,  \\[4pt]
 u =0 \quad\text{on } \pd \Om
 \end{array}
\right.
\end{equation}
and
\begin{equation}
\label{bvp-dual}
\left\{\begin{array}{rr}
 -\De v - \bb \cdot \nb v  +cv= g \quad \text{in }\Om ,\,\,\,\, \\[4pt]
 v =0 \quad\text{on } \pd \Om ,
 \end{array}
\right.
\end{equation}
where  $\mathbf{b} =(b_1 , ..., b_n )  $ and $c  $ are given   functions on $\Omega$ belonging to the critical weak spaces $L^{n,\infty} (\Omega ; \R^n )$ and $L^{n/2,\infty} (\Omega )$, respectively. Here for $1 \le  p< \infty$ and $1 \le q \le \infty$, $L^{p,q}(\Omega )$ is  the Lorentz space whose  quasi-norm is denoted by $\|\cdot \|_{L^{p,q}(\Omega )}$, $\|\cdot \|_{p,q;\Omega}$,  or  simply  $\|\cdot \|_{p,q}$. Recall that $L^{p,p}(\Omega ) = L^p (\Omega )$; so if $p=q$, we write $\|\cdot \|_p =\|\cdot \|_{p,p}$.

\medskip

There is a vast literature on the existence, uniqueness, and regularity of solutions of   second order elliptic PDEs of the form
\[
- \pd_i (a_{ij} \pd_j u -  \td b_i u)  - \bb \cdot\nb u + cu = f
\]
and their  variants such as non-divergence form, systems, and   parabolic counterparts.
A few references can be found in the classical books \cite{GilTru,LU}, in \cite{DK}, and in papers cited in \cite{KiTs}. In this paper, motivated by the applications to fluid dynamics,
we  search for minimum assumptions made on the lower-order coefficients $\bb$ and $c$. Since the regularity of $a_{ij}$ is not our focus, we assume that  $a_{ij}=\de_{ij}$ for simplicity.

Existence, uniqueness, and regularity of weak solutions in $W^{1,p}(\Omega)$ or  $W^{2,p}(\Omega)$, $1<p<\infty$, of \eqref{bvp} and \eqref{bvp-dual} have been well known for sufficiently regular $\bb$ and $c$;
for instance, see  \cite[Theorem 9.15]{GilTru}  for $|\bb|,c\in L^\infty(\Omega)$, and \cite[Chap.~III, Theorem 15.1]{LU} for more general $\bb$ and $c$ satisfying
\[
 \bb  \in L^q(\Om; \R^n ), \quad c \in L^{q/2}(\Om) \quad\mbox{for some}\,\,  q>n.
\]
In this subcritical case, the lower order terms may be treated as perturbations of  the leading term $-\De u$.
See also a recent paper \cite{KRW} for  existence and uniqueness results in mixed-norm parabolic  Sobobev $W^{1,p}$-spaces for the corresponding parabolic equations in which the lower order coefficients  $\tilde{b}_i, b_i$, and $c$ are in suitable subcritical mixed-norm Lebesgue spaces.

In this paper, the coefficients $\bb$ and $c$  belong to \emph{critical}  spaces, that is,
\[%
\bb \in L^n(\Om ; \R^n), \quad c \in L^{n/2}(\Om)
\]
or more generally
\[
\bb \in L^{n,\infty}(\Om; \R^n), \quad c \in L^{n/2,\infty}(\Om),
\]%
which prevents us from treating the lower order terms as  perturbations. Another viewpoint is to consider the rescaled functions
\begin{equation}\label{scaling-0}
u_R (x)=u(Rx), \quad \bb_R (x) = R \bb (Rx), \quad c_R (x)= R^{2} c(Rx),
\end{equation}
where $R $ is a positive constant.
If $u$ solves \eqref{bvp} in $B_R$ with coefficients $\bb$ and $c$, then
$u_R$ solves \eqref{bvp} in $B_1$ with coefficients $\bb_R$ and $c_R$, and
$$
\norm{\bb_R}_{L^{n,\infty}(B_1)}=\norm{\bb}_{L^{n,\infty}(B_R)} , \quad \norm{c_R}_{L^{n/2,\infty}(B_1)}=\norm{c}_{L^{n/2,\infty}(B_R)},
$$
where $B_r =B_r (0)$, and  $B_r (x_0) = \{ x \in \R^n: |x-x_0| < r\}$ is the open ball in $\mathbb{R}^n$ centered at $x_0\in \R^n$ with radius  $r>0$. Hence $L^{n,\infty}$ and $L^{n/2,\infty}$ are scale-invariant spaces for $\bb$ and $c$, respectively, with respect to the scalings in  (\ref{scaling-0})

For general  $\bb \in L^{n,\infty} (\Omega ; \R^n )$, the  problem \eqref{bvp} may have no weak solutions in $W_0^{1,2}(\Omega )$, as shown by the following example.

\begin{example}\label{eg:noncoercive-LnI} Consider the problems  \eqref{bvp} and  \eqref{bvp-dual}, where
$$
\Omega= B_1 , \quad \bb (x)= - \frac{M x}{|x|^2}, \quad\mbox{and}\quad c =0  .
$$
Assume that $ M>  (n-2)/2$ and $M \neq n-2$.
Then $v (x) =|x|^{M-n+2} - 1$ is a weak solution in $W_0^{1,2}(\Omega )$ of \eqref{bvp-dual} with the trivial data $g=0$. This shows that uniqueness fails to hold for weak solutions in $W_0^{1,2}(\Omega )$  of the dual problem \eqref{bvp-dual}. By a duality argument (see \cite{Mo}), there exists  $f \in W^{-1,2}(\Omega )$ such that  the problem \eqref{bvp}   has no weak solutions in $W_0^{1,2}(\Omega )$
It was also observed in \cite[Section 7]{Mo} (see also   \cite[Example 1.1]{KiTs}) that if $2<p<n$ and $(n-p)/p \le M < (n-2)/2$, then there are  no weak solutions of  \eqref{bvp} in $W^{1,p}_0(\Omega )$  for some $f \in W^{-1,p}(\Omega )$.
It should be noted that
$\bb \in L^{n,\infty} (\Omega ; \R^n ) \setminus L^n (\Om ; \R^n )$.
\end{example}

\medskip

Example \ref{eg:noncoercive-LnI} suggests us to impose some additional condition on the drift $\bb$ for existence or uniqueness of weak solutions of the problems \eqref{bvp} and \eqref{bvp-dual}. Note that if $\bb (x)= - M x /|x|^2$ for some $M>0$, then ${\rm div}\, \bb (x) = -M (n-2) /|x|^2$ and   $\inf_{B_1 (0)  \setminus \{ 0\}}  {\rm div}\, \bb =-\infty$. Hence such an example may be excluded by assuming that ${\rm div}\, \bb \ge  - C$  in $\Omega$ for some constant $C \ge 0$.

In general,
lower order terms with critical coefficients
can be ``controlled'' in a few cases. The first case is when the coefficients have small sizes,  for example, when $\norm{\bb}_{n,\infty}$ and $\norm{c}_{n/2,\infty}$ are sufficiently small (see e.g.~\cite{Kry}  which indeed assumes smallness conditions on $\bb$ and $c$ in Morrey spaces).  Second, when
$\bb \in L^n(\Om; \R^n)$ and $c \in L^{n/2}(\Om)$ or more generally when  $\bb \in L^{n,q}(\Om; \R^n )$ and $c \in L^{n/2,q}(\Om)$ for some $q<\infty$, the norms become small over sufficiently small balls
although they may be large in the entire domain $\Om$. This approach has been  taken in Droniou \cite{Dr}, Moscariello \cite{Mo},
Kim and Kim \cite{KK}, and Kang and Kim \cite{KangK}.
 Finally, if $\bb \in L^{n,\infty}(\Om; \R^n )$ and its norm is not small, the term $\div(u\bb)$ may still be controlled by  using the coercivity of the bilinear form associated with \eqref{bvp} if we assume that $\div \bb = 0$ (Zhikov \cite{Zhi}, Kontovourkis  \cite{Ko}, Zhang \cite{ZhangQi}, Chen et al.~\cite{CSTY}, Seregin et al. \cite{SSSZ}, Filonov \cite{Fi}, Ignatova, Kukavica, and Ryzhik \cite{IKR}, and Filonov and Shilkin \cite{FSh1,FSh2}; in this case $L^\infty$ bound, but not H\"older continuity, may be obtained under weaker integrability condition of $\bb$),
or if we assume that   $\div \bb \ge 0$ (Nazarov and Uraltseva \cite{NU},
Kim and Tsai \cite{KiTs},  Kwon \cite{Kwon},  and Chernobai and Shilkin \cite{CS}).

One may try to combine the first two approaches, by observing that if either $\norm{\bb}_{n,\infty}\le \eps \ll1$ or $\bb \in L^{n}(\Om; \R^n )$, then
\[
\norm{\bb}_{n,\infty, (r)} \le \eps  \quad \text{for some }r>0,
\]
where the   small scale quasi-norms $\norm{ b}_{p,\infty, (r)}$ on $L^{p,\infty}(\Omega )$  are  defined as
\[
\norm{ b}_{p,\infty, (r)}= \norm{ b}_{p,\infty, (r); \Omega}= \sup_{x \in \Om} \norm{b}_{L^{p,\infty}( \Omega \cap B_r (x) )}
\]
for $r>0$.
It is obvious that $\norm{ b}_{p,\infty, (r)} \le \norm{ b}_{p,\infty}$ for each $r>0$.
Moreover, since $\Omega$ is bounded, there exist $N$ points $x_1 , ..., x_N$ in $\Om$ with $N \le C(n,   \Omega , r)$  such that $\Omega \subset \cup_{j=1}^N B_r (x_j )$ and so
\[
\norm{ b}_{p,\infty} \le N  \sum_{j=1}^N  \norm{ b}_{L^{p,\infty}(\Omega \cap B_r(x_j ) )} \le  N^2 \norm{ b}_{p,\infty, (r)} .
\]
Hence $\norm{\cdot}_{p,\infty, (r)}$ is an equivalent quasi-norm on $L^{p,\infty} (\Omega    )$ for any $r>0$.
 If $\norm{ b}_{p,\infty}$ is small, so is $\norm{ b}_{p,\infty, (r)}$. But since the number $N$ depends on $r$ in general, $\norm{ b}_{p,\infty}$ may be large although $\norm{ b}_{p,\infty, (r)}$ is small.

\begin{example}\label{ex2} Let $1 < p<\infty$.
For $\e>0$ and  $0<r<1$,    we define
\[
b(x) = \sum_{ k \in \mathbb Z^n} \frac {\e}{|x-2rk |^{n/p}} \mathbf{1}_{B_r (2rk)} (x)   \quad (x\in \R^n ),
\]
where  $\mathbf{1}_{A}$ denotes the characteristic function of  a set  $A$. Then it can be shown (see Example \ref{ex2-proof}) that
$$
\norm{ b}_{p,\infty, (r);B_1} \approx  \e \quad\mbox{and}\quad
\norm{ b}_{p,\infty; B_1} \approx   \e r^{-n/p}
$$
for small $r>0$.
\end{example}

Motivated by the above consideration, we will henceforth make the following assumptions on the   drift $\bb$ and the   coefficient   $c$:
\begin{equation}\label{b1 and b2}
\bb =\bb_1 +\bb_2 , \quad \bb_1 , \, \bb_2 \in  L^{n,\infty} (\Omega ; \R^n ) , \quad c \in L^{n/2,\infty}(\Omega) ,
\end{equation}
\begin{equation}\label{coercivity cond-b1}
  \div \mathbf{b}_1  \ge 0 , \,\, c \ge 0 \quad\mbox{in} \,\, \Om ,
 \quad\mbox{and}\quad \norm{\bb_2}_{n,\infty, (r)}\le \e
\end{equation}
for some $r>0$, where  $\e = \e(n, \Om) $ is a sufficiently small positive number.
Sometimes, in addition to (\ref{b1 and b2}) and (\ref{coercivity cond-b1}),  it will be  assumed that
\begin{equation}\label{coercivity cond-b1-stronger}
\div \mathbf{b}_1, \div \mathbf{b}_2 \in L^{n/2,\infty} (\Omega )  ,\qquad
\norm{\div \bb_2}_{n/2,\infty, (r)}\le \e.
\end{equation}
When we assume \eqref{coercivity cond-b1-stronger}, we may extract $\bb_3 \in L^n(\Om)$ from $\bb_2$ and make no assumption on $\div \bb_3$.
Notice that even the case $\bb_1=0$, $c=0$,  $ \norm{\bb_2}_{n,\infty, (r)} \le \e$   has not been studied in the literature yet.
In this paper, we show that   the smallness of $\norm{\bb_2}_{n,\infty, (r)}$ is still sufficient to get existence, uniqueness, and regularity results for weak and strong solutions of \eqref{bvp} and \eqref{bvp-dual}. These results are stated in Theorems \ref{th4-q version}, \ref{th4-q version-strong}, and \ref{th4-q version-strong-u}. Furthermore, higher integrability estimates for the gradient  of a solution $v$  of the dual problem \eqref{bvp-dual} are   obtained in Theorem  \ref{th4}, which are deduced  from  global H\"{o}lder regularity estimates of the solution $v$. Uniqueness of very weak solutions of  \eqref{bvp}  that are slightly below the threshold is also proved in Theorem \ref{th5}.

\medskip

The paper is organized as follows. In Section \ref{Sec2}, we state all the main results in the paper (Theorem \ref{th4-q version}, \ref{th4-q version-strong}, \ref{th4-q version-strong-u}, \ref{th4}, and \ref{th5}). The approaches to prove these results are  outlined in this section. Section \ref{Sec3} is devoted to stating and proving preliminary results for Lorentz spaces, some estimates involving weak quasi-norms, mollification in Lorentz spaces, and  the Miranda-Nirenberg interpolation inequality.
Proofs of Theorems \ref{th4-q version}, \ref{th4-q version-strong}, and \ref{th4-q version-strong-u}  are  provided in  Sections \ref{Sec4} and    \ref{Sec5}.  Section \ref{Sec6} is fully devoted to proving   global H\"older estimates for weak solutions of \eqref{bvp-dual}, which is a main ingredient to prove Theorem  \ref{th4}. Finally, in Section \ref{Sec7}, we complete the proofs of   Theorems \ref{th4} and   \ref{th5}.

\section{Main results} \label{Sec2}

 Throughout the paper, for any given number $p \in (1, \infty)$, we denote by $p'$ the H\"{o}lder conjugate of $p$, i.e., $p'= p/(p-1)$. In addition, for $p \in [1,n)$, let $p^*$ denote the Sobolev conjugate of $p$, precisely
$p^* = np/(n-p).$

Let $\Om$ be a bounded  Lipschitz   domain in $\R^n$, where $n \ge 3$. Then for  $n' <  p<n$, we have the following well-known estimates (see  \cite[Lemma 3.6]{KiTs} e.g., and Lemmas \ref{th3-1} and \ref{th1-2}):
$$
\norm{u \mathbf{b}}_{p } \le C
\norm{\mathbf{b}}_{n,\I }\norm{u}_{W^{1,p}(\Om)}
$$
and
\[
\norm{c u}_{W^{-1,p}(\Om ) } \le C  \norm{c}_{n/2,\infty} \norm{u}_{W^{1,p}(\Om)}
\]
for all $u \in W^{1,p}(\Omega )$, where  $C=C (n, p,\Om )$.  Hence it makes sense to define weak solutions of \eqref{bvp} as follows.
\begin{definition}\label{def2.1} %
Let $\mathbf{b} \in L^{n,\infty}(\Om;\R^n)$ and  $c \in L^{n/2, \infty}(\Om )$. Assume that $f\in W^{-1,p}(\Om)$  and $n' <p<n$. Then a function $u \in W_0^{1,p}(\Om )$ is called a \emph{weak solution} in $W_0^{1,p}(\Om )$ or a $p$-\emph{weak solution} of \eqref{bvp} if it satisfies
\begin{equation}\label{eq1.3}
\int_\Om  \left[ \bke{\nb u - u \mathbf{b}} \cdot \nb
\phi  + c u \phi \right]   dx = \bka{f,\phi} \quad \text{for all } \phi \in W_0^{1,p'}(\Om).
\end{equation}
Weak solutions in $W_0^{1,2}(\Om )$ of \eqref{bvp} are simply called \emph{weak solution}s.
In addition, a $p$-weak solution $u$ of \eqref{bvp} will be called a \emph{strong solution} if it satisfies $  u \in W_{loc}^{2,1}(\Om )$. Weak and $p$-weak solutions of the dual problem \eqref{bvp-dual} can be similarly defined.
\end{definition}

\medskip
The first purpose of the paper is to establish  existence and uniqueness results for $p$-weak solutions  (Theorem \ref{th4-q version})  and strong solutions  (Theorems \ref{th4-q version-strong} and \ref{th4-q version-strong-u})  of the problem \eqref{bvp} and its dual \eqref{bvp-dual}.

\begin{theorem} \label{th4-q version}
Let $\Om$ be a bounded  $C^1$-domain in $\R^n$ with $n \ge 3$, and let $p \in (n', n)$ and $M \in (0,\infty)$. Then there exists  a small number $\varepsilon >0 $, depending only on $n,   \Omega, p$,   and $M$,  such that the following statements hold:

\noindent
Assume  that
\[
\begin{split}
&\bb = \bb_1 + \bb_2   , \quad (\mathbf{b}_1 , \mathbf{b}_2 ) \in L^{n,\I}(\Om;\R^{2n}), \quad   c \in L^{n/2 ,\infty}(\Om ), \\
&\|\bb_1\|_{n, \infty} + \norm{c}_{n/2,\infty} \le M ,\quad\mbox{and}\quad \div \mathbf{b}_1 \ge 0 , \, c \ge 0\,\,\mbox{in}\,\,\Om.
\end{split}
\]
 If $n' <  p<2$, assume further that
\[
\begin{split}
&\div \bb_1  \in L^{n/2,\I} (\Om ), \quad \|\div \mathbf{b}_1\|_{n/2,\infty} \le M, \quad \bb_2 =\bb_{21} + \bb_{22}, \\
&\mathbf{b}_{21} \in L^{n}(\Om;\R^n),\quad\mbox{and}\quad   {\rm div}\, {\mathbf b_{22}} \in L^{n/2,\I} (\Om ).
\end{split}
\]
Assume also  that  $\bb_2$ satisfies
\[
\norm{\bb_2}_{n,\infty, (r)} +%
\mathbf{1}_{\{p<2\}} \left(\norm{\bb_{22}}_{n,\infty, (r)} + \norm{\textup{div} \,\bb_{22}}_{n/2,\infty, (r  )}  \right) \le \e
\]
for some $r \in (0,  \textup{diam}\, \Omega )$. Then: %

\begin{itemize}
\item[\textup{(i)}] For each $f \in W^{-1,p}(\Om  )$, there exists a unique $p$-weak solution $u $ of \eqref{bvp}. Moreover, we have
$$
 \| u \|_{W^{1,p} (\Om )} \le C   \|f\|_{W^{-1,p}(\Om )}.
$$
\item[\textup{(ii)}] For each $g \in W^{-1,p'}(\Om )$, there exists a unique $p'$-weak solution $v $ of \eqref{bvp-dual}. Moreover, we have
$$
 \|  v \|_{W^{1,p'} (\Om )} \le C   \|g\|_{W^{-1,p'}(\Om )}.
$$
\end{itemize}
Here the constant $C>0$ depends only on $n, \Omega, p,  r$,  $M$, $ \| {\mathbf b}\|_{2}$,    and $\bb_{21}$.
\end{theorem}
\noindent
\begin{remark}
The condition $\bb_{21} \in L^{n}$
when $n'<p<2$ is only for simplicity of presentation, and can be relaxed to $\bb_{21} \in L^{n,q}$ for some $1\le q<\infty$. Moreover, the dependence of $C$ on $\bb_{21}$ can be made explicit, so that it is only through $\norm{\bb_{21}}_{n,q}$ and the length scale $\rho$ such that the $\rho$-mollification of $\bb_{21}$ well approximates $\bb_{21}$ in  $L^{n,\infty}$. See Proposition \ref{prop4.4} for the detailed statement.
\end{remark}
The following two theorems are $W^{2,q}$-versions of Theorem \ref{th4-q version} for $v$ and $u$, respectively. The stronger assumption of Theorem \ref{th4-q version-strong-u} means $\bb_{21}=0$ and $\bb_2=\bb_{22}$; see Remark \ref{rem51} after its proof.

\begin{theorem}\label{th4-q version-strong}
Let $\Om$ be a bounded  $C^{1,1}$-domain in $\R^n$ with $n \ge 3$, and let $q \in (1, n/2)$  and $M \in (0,\infty)$. Then there exists  a small number $\varepsilon >0 $, depending only on $n,   \Omega, q$, and $M$,   such that the following statement holds:

\noindent
Assume  that
\[
\begin{split}
&\bb = \bb_1 + \bb_2  , \quad (\mathbf{b}_1 , \mathbf{b}_2 ) \in L^{n,\I}(\Om;\R^{2n}), \quad  (\div \mathbf{b}_1  , c ) \in L^{n/2 ,\infty}(\Om ;\R^2 ),\\
&\|\bb_1\|_{n, \infty}   + \norm{(\div \mathbf{b}_1  , c )}_{n/2,\infty} \le M , \quad\mbox{and}\quad \div \mathbf{b}_1 \ge 0 , \, c \ge 0 \,\,\mbox{in}\,\, \Om.
\end{split}
\]
If $2n/(n+2) <  q< n/2$, assume further  that
$$
\bb_2 =\bb_{21} + \bb_{22} , \quad\mathbf{b}_{21} \in L^{n}(\Om;\R^n), \quad\mbox{and}\quad   {\rm div}\, {\mathbf b_{22}} \in L^{n/2,\I} (\Om ).
$$
Assume also  that  $\bb_2$ satisfies
$$
\norm{\bb_2}_{n,\infty, (r)} + \mathbf{1}_{\{q   >2n/(n+2) \}}  \left(\norm{\bb_{22}}_{n,\infty, (r)} + \norm{\textup{div} \,\bb_{22}}_{n/2,\infty, (r  )}  \right) \le \e
$$
for some $r \in (0,  \textup{diam}\, \Omega )$.

 Then for each $g\in L^{q}(\Om)$, there exists a unique $q^*$-weak solution $v $ of \eqref{bvp-dual}. Moreover, we have
$$
 v \in W^{2,q} (\Om )
 \quad\mbox{and}\quad
  \|  v\|_{ W^{2,q} (\Om ) } \le C   \|g\|_{L^q (\Omega )}
$$
for some constant $C =C( n, \Omega , q,r ,   M ,    \|   {\mathbf b}\|_{2}   , \bb_{21})>0$.
\end{theorem}

\begin{theorem}\label{th4-q version-strong-u}
Let $\Om$ be a bounded  $C^{1,1}$-domain in $\R^n$ with $n \ge 3$, and let $q \in (1, n/2)$ and $M \in (0,\infty)$. Then there exists  a small number $\varepsilon >0 $, depending only on $n,   \Omega ,q$,   and $M$,  such that the following statement holds:

\noindent
 Assume that
\[
\begin{split}
&\bb = \bb_1 + \bb_2,\quad  (\mathbf{b}_1 , \mathbf{b}_2 ) \in L^{n,\I}(\Om;\R^{2n}), \quad (\div \mathbf{b}_1, \div \mathbf{b}_2 ,  c )   \in L^{n/2 ,\infty}(\Om ;\R^3 ), \\
& \|\bb_1\|_{n, \infty}   + \norm{(\div \mathbf{b}_1  , c )}_{n/2,\infty} \le M, \quad \mbox{and} \quad  \div \mathbf{b}_1 \ge 0, \,\, c \ge 0\,\,\mbox{in}\,\,\Om.
\end{split}
\]
Assume also that   $\bb_2$ satisfies
$$
\norm{\bb_2}_{n,\infty, (r)} +   %
  \norm{\textup{div} \,\bb_{2}}_{n/2,\infty, (r  )}  \le \e
$$
for some $r \in (0,  \textup{diam}\, \Omega )$.

Then for each $f\in L^{q}(\Om)$, there exists a unique $q^*$-weak solution $u$ of \eqref{bvp}. Moreover, we have
$$
 u \in W^{2,q} (\Om )
 \quad\mbox{and}\quad
  \|  u\|_{ W^{2,q} (\Om ) } \le C   \|f\|_{L^q (\Omega )}
$$
for some constant  $C =C (n,  \Omega , q,r,   M ,   \| {\mathbf b}\|_{2})>0$.
\end{theorem}

\begin{remark} If ${\mathbf b}_2 \in L^{s}(\Om ;\R^n )$ for some $s>n$, then the constant $C$ depends on the norm of ${\mathbf b}_2$; see Remark \ref{rem-cons-depend}.
\end{remark}

\begin{remark}
Assume that  $\bb  \in L^{n,\I}(\Om;\R^n)$, ${\rm div}\, {\mathbf b}  \in L^{n/2,\I} (\Om )$, $\div \mathbf{b}  \ge -K$ in $\Om$, and $K$ is a positive constant. Then since $\bb$ can be written as  $\bb = \bb_1 +\bb_2$, where $\bb_1 =  \bb -Kx/n $ and $\bb_2 =Kx/n$, both Theorems \ref{th4-q version} and  \ref{th4-q version-strong} hold with the constant $C$ depending on $  \|{\mathbf b} \|_{n,\I}$, $\|{\rm div}\, {\mathbf b} \|_{n/2,\I}$, and $K$.
\end{remark}

\medskip

 The second purpose of the paper is to establish
 $W^{1,n+\delta_1}$- or $W^{2,n/2+\delta_2}$-regularity of weak solutions of the dual problem \eqref{bvp-dual} for some $\delta_1, \delta_2>0$.

\begin{theorem}\label{th4} Let $\Om$ be a bounded  $C^{1,1}$-domain in $\R^n$ with $n \ge 3$, and let $p \in (n,\infty)$,  $q \in (n/2,\infty)$, and $M \in (0,\infty)$. Then  there exists  a small number $\varepsilon  >0$, depending only on $n,  \Omega ,p,q$,    and $M$,  such that the following statements hold:

\noindent
Assume that
\begin{equation}\label{th4-assumption}
\begin{split}
& \bb = \bb_1 + \bb_{2} +\bb_{3}, \quad (\mathbf{b}_1 , \mathbf{b}_2) \in L^{n,\I}(\Om;\R^{2n}), \quad \mathbf{b}_{3} \in L^{n}(\Om;\R^n),\\
& c \in L^{p^{\sharp}}(\Omega)\,\,\mbox{with}\,\, p^{\sharp} = \frac{np}{n+p},\quad (\div \mathbf{b}_1 ,  {\rm div}\, {\mathbf b_{2}})\in L^{n/2 ,\infty}(\Om ;\R^2 ), \\
&    \|\bb_1\|_{n, \infty}   + \norm{ \div \mathbf{b}_1  }_{n/2,\infty} \le M ,  \quad\div \mathbf{b}_1 \ge 0 , \,\,c \geq 0 \,\,\mbox{in}\,\,\Om  ,
\end{split}
\end{equation}
and
\begin{equation}\label{th4-assumption=2}
  \norm{\bb_{2}}_{n,\infty, (r  )} + \norm{\textup{div} \,\bb_{2}}_{n/2,\infty, (r  )}  \le \e
\quad\mbox{for some}\,\,   r \in (0,  \textup{diam}\, \Omega ).
\end{equation}
Then  for each $g \in W^{-1, 2}(\Omega)$, there exists  a unique weak solution $v \in W^{1,2}_0(\Omega)$ of \eqref{bvp-dual}. Moreover, this   solution $v$ has the following regularity properties:
\begin{itemize}
\item[\textup{(i)}] If $g \in W^{-1,p}(\Om)$, then
$$
v \in W^{1,n+\delta_1}_0(\Om) \quad\mbox{and}\quad \|v\|_{W^{1,n+\delta_1}_0(\Om)} \le C \|g\|_{W^{-1,p} (\Omega )}
$$
for some $\delta_1   \in (0, p-n]$ and $C>0$ depending only on $n$,   $\Omega $, $p$, $r$, $M$, $\|\bb\|_{n,\infty}$, $\bb_3$, and $\|c\|_{p^{\sharp}}$.

\item[\textup{(ii)}] If $g \in L^{q}(\Om)$, then
$$
v \in W^{2,n/2+\delta_2}(\Om) \quad\mbox{and}\quad \|v\|_{W^{2,n/2+\delta_2}(\Om)} \le C \|g\|_{L^q (\Omega )}
$$
for some $\delta_2   \in (0,  q-n/2]$ and $C>0$ depending only on $n$,    $\Omega $,  $p$, $q$, $r$, $M$, $\|\bb\|_{n,\infty}$, $\bb_3$, and $\|c\|_{p^{\sharp}}$.
\end{itemize}
\end{theorem}

By the Morrey embedding theorem, the estimates in Theorem \ref{th4} imply H\"older estimates for solutions of \eqref{bvp-dual}. However, their proof start  with H\"older estimates in Theorem \ref{thm-0322}.

\medskip

As an important consequence of Theorem \ref{th4}, we   prove existence and uniqueness results for $p$-weak solutions or very weak solutions in $L^q (\Omega )$ of \eqref{bvp}, where $ p< n/(n-1)$ and $q< n/(n-2)$. Note that
$$
n' =\frac{n}{n-1} \quad\mbox{and}\quad (n')^* = \frac{n}{n-2} = \left( \frac n2 \right)' .
$$
For the simplicity of presentation, let us define
\[
W_0^{1,p-}(\Om )=\bigcap_{q<p}W_0^{1,q}(\Om )   \quad\mbox{and}\quad W^{-1,p-}(\Om )=\bigcap_{q<p}W^{-1,q}(\Om ).
\]

\begin{theorem}\label{th5} Let $\Om$ be a bounded  $C^{1,1}$-domain in $\R^n$ with  $n \ge 3$, and let $p \in (n, \infty )$ and $M \in (0,\infty )$.  Then  there exists a small number $\varepsilon  >0$, depending only on $n$,  $\Omega$, $p$, and $M$,  such that the following statements hold:

\noindent
Assume that $(\bb , c)$ satisfies the same assumptions \eqref{th4-assumption} and \eqref{th4-assumption=2} as   Theorem \ref{th4}.
 Then:

\begin{itemize}
 \item[\textup{(i)}] There exists   $l_0   \in (n ',(n/2)')$,   close to $(n/2)'$, such that if $u \in L^{l_0} (\Omega )$ satisfies
\begin{equation}\label{eq1.3-0}
\int_\Om u \left( - \De \phi - \mathbf{b} \cdot \nb
\phi +c \phi \right) dx = 0 ,
\end{equation}
for all $\phi \in C^2 (\Omega) \cap C^{1,1} (\overline{\Om})$ with $\phi|_{\partial\Omega}=0$, then $u=0$ identically on $\Omega$.
 \item[\textup{(ii)}]
For each $f \in W^{-1, n'-}(\Om)$ there exists a unique weak solution $u$ in $W_0^{1,n'-}(\Om )$ of \eqref{bvp}.
\end{itemize}
\end{theorem}
Let us now  summarize our approach to prove the main results. To prove   existence of $p$-weak solution with $p \in [2,n)$, we begin with noting  that the bilinear form associated with \eqref{bvp}, that is,
\[
B(u,v)= \int_\Om \left[ \bke{\nb u - u \mathbf{b}} \cdot \nb
v  + c uv \right] dx
\]%
is bounded on $W_0^{1,2}(\Omega )\times W_0^{1,2}(\Omega )$. However, under assumptions (\ref{b1 and b2}) and (\ref{coercivity cond-b1}) on $\bb$ and $c$, there is no sign condition on $\div \mathbf{b}$. Consequently, $B$ fails to be coercive. Hence the  existence of weak solutions of \eqref{bvp} cannot be deduced from the Lax-Milgram theorem. To overcome this difficulty, we apply  the method of continuity, the key step of which is   to derive the following a priori estimate for   $p$-weak solutions $u$ of  \eqref{bvp}:
\begin{equation}\label{apriori estimate-p}
 \| \nabla u \|_{p}   \le C    \|f\|_{_{W^{-1,p}(\Om )}} ,
\end{equation}
where $C$ is a positive constant independent of $f$ and $u$. To prove (\ref{apriori estimate-p}), we
observe that
\[
-\De u +   \div( u \mathbf{b}_1 ) +   cu =f -  \div( u \mathbf{b}_2 ) \quad\mbox{in}\,\,\Omega .
\]
Then since $(\bb_1 , c)$ satisfies the condition of  \cite[Theorem 2.1]{KiTs}, there exists a constant $C_1 >0$ such that
\[
 \| \nabla u \|_{p}   \le C_1  \left( \|f\|_{_{W^{-1,p}(\Om )}}   +  \|  u \mathbf{b}_2\|_{p}\right) .
\]
By a bilinear estimate (Lemma \ref{th3-2}) involving the new quasi-norm $\norm{\bb_2}_{n,\infty, (r)}$,   the problematic term $\|  u \mathbf{b}_2\|_{p}$ can be replaced     by   $\|u\|_p$, under the  smallness condition in (\ref{coercivity cond-b1}). Finally, the   term $\|u\|_p$ is removed to obtain (\ref{apriori estimate-p})  by  using a quite standard   estimate for the distribution function of   $u$ (see Lemma \ref{log-estimate}).

Applying the method of continuity as outlined above,  we show  that if $2 \le p< n $, then   for each $f \in W^{-1,p}(\Omega )$ there exists a unique $p$-weak solution $u$ of \eqref{bvp}. By a standard duality argument, it then follows  that  for each $g \in W^{-1,p'}(\Omega )$ there exists a unique $p'$-weak solution $v$ of \eqref{bvp-dual}. These results are proved by assuming that  $(\bb, c)$ satisfies (\ref{b1 and b2}) and (\ref{coercivity cond-b1}).  To obtain similar results for the case $p<2$, we need to make  an additional assumption on $\bb$.
Suppose in addition to (\ref{b1 and b2}) and (\ref{coercivity cond-b1}) that  $\bb$ satisfies (\ref{coercivity cond-b1-stronger}). Then since the equation in  \eqref{bvp-dual} can be written as
\[
-\De v - \bb_1 \cdot \nb v  +cv=   g +{\rm div}\, (v \bb_2 ) - ({\rm div}\, \bb_2 )v ,
\]
we can derive the a priori estimate
\[
\|\nabla v\|_{p'} \le C \|g\|_{W^{-1,p'}(\Omega )}
\]
for $n/(n-1)<p<2$, by using a bilinear estimate involving the functional $M_r (\bb_2 ) = \norm{\bb_2}_{n,\infty,(r)} + \norm{\div \bb_2}_{n/2,\infty,(r)}$ (see Lemma \ref{th3-3}). Hence by the method of continuity  and then by duality, we deduce that if $n/(n-1)<p<2$, then for each $g \in W^{-1,p'}(\Omega )$ there exists a unique $p'$-weak solution $v$ of \eqref{bvp-dual}, and  for each $f \in W^{-1,p}(\Omega )$ there exists a unique $p$-weak solution $u$ of \eqref{bvp}.
Moreover, it will be shown that if $f\in L^q (\Om )$ and $1<q < n/2$, then a weak solution $u$ of \eqref{bvp} has the strong $L^q$-regularity, that is, $u  \in W^{2,q}(\Om )$. A similar regularity result also holds for weak solutions   of \eqref{bvp-dual} under a slightly more general condition on $\bb$. See the proofs of Theorems \ref{th4-q version}, \ref{th4-q version-strong}, and \ref{th4-q version-strong-u} for complete details.

After proving Theorems \ref{th4-q version}, \ref{th4-q version-strong}, and \ref{th4-q version-strong-u}, the remaining  part of the paper is mainly devoted to studying  further regularity of a  weak solution $v$   of \eqref{bvp-dual}. Assume that   $g \in W^{-1, p}(\Omega )$ and  $n< p < \infty $. Then by Theorem \ref{th4-q version},  there exists a unique weak solution $v$ of \eqref{bvp-dual} and
$v$ belongs to $W^{1,q}(\Omega )$ for any $q<n$.
It is   well-known (see \cite{KangK,KK} e.g.) that if $\bb \in L^n (\Omega: \R^n )$, then  $v \in  W^{1,p}(\Omega )$. However for general $\bb$ in $L^{n,\infty} (\Omega: \R^n )$, only  partial regularity of $v$ has been  proved, for instance, in \cite[Theorem 2.3]{KiTs}. Extending this result   to a more general class of drifts $\bb$ satisfying (\ref{b1 and b2}), (\ref{coercivity cond-b1}), and  (\ref{coercivity cond-b1-stronger}), we show in Theorem \ref{th4} that  if $c \in L^{p^\sharp} (\Omega )$, where $p^\sharp =np/(n+p)$, then  $v \in W^{1,n+\delta_1 }(\Omega )$ for some $\delta_1 >0$. It is also shown that if $g \in L^q (\Omega )$ for some $q >n/2$, then  $v \in W^{2,n/2+\delta_2 }(\Omega )$ for some $\delta_2 >0$. The key step of our proof of Theorem \ref{th4} is to prove the global H\"older regularity of $v$, by applying the De Giorgi iteration method.
Then making use of the Miranda-Nirenberg interpolation inequality as in \cite{KiTs}, we conclude that   $v \in W^{1,n+\delta_1 }(\Omega )$ or $v \in W^{2,n/2+\delta_2 }(\Omega )$.

Finally, by duality arguments based on Theorem \ref{th4}, we   prove uniqueness and existence results (Theorem \ref{th5}) for     weak and very weak solutions of   \eqref{bvp}, which cannot be covered by Theorem  \ref{th4-q version}.

\section{Preliminaries}
\label{Sec3}

For nonnegative  quantities $a$ and $b$, we write $a \lesssim b$ if there exists a positive constant $C$ such that $a \le C b$. If $a \lesssim b$ and $a \lesssim b$, we write $a\approx b$.

\subsection{Lorentz spaces}

Let $\Omega$ be any domain in $\R^n$. For a Lebesgue measurable function  $f$ on $\Om$, let $f^*$ be the decreasing rearrangement of $f$  defined by
\[
f^* (t) = \inf \left\{ \lambda \ge 0 \,:\, \mu_f (\lambda ) \le t \right\}  \quad (t \ge 0),
\]
where $\mu_f$  is   the   distribution function  of $f$:
\[
\mu_f (\lambda) = | \{ x \in \Om : |f(x)|>\lambda \}| \quad   (\lambda \ge 0).
\]
Then for  $0 < p\le\infty$  and $0< q \le \infty$, the Lorentz space $L^{p,q}(\Om )$ is a quasi-Banach space equipped with the quasi-norm
\begin{equation*}
\|f \|_{L^{p,q}(\Om )}= \left\{
\begin{aligned}
 &\left(  \int _0 ^\infty \left[ t^{1/p} f^* (t) \right]^q   \frac{d t}{t} \right )^{1/q}\quad& \text{when} \,\,  q<\infty,\\
 	&\quad \sup_{t >0} \left[t^{1/p} f^* (t) \right] \quad& \text{when} \,\, q=\infty.
 \end{aligned}
 \right.
\end{equation*}
It is well-known  (see \cite{AdamsFournier,BL,BS,Gr1} e.g.) that
if $0<p \le \infty$, then $L^{p,p}(\Omega )=L^p (\Omega )$; and if $0 < p \le \infty$ and $0< q_1 \le q_2 \le \infty$, then $L^{p  , q_1 }(\Omega )\subset L^{p,q_2 }(\Omega )$. For simplicity,  we often write $\|f\|_{p,q} =\|f\|_{L^{p,q}(\Om )}$ and $\|f\|_{p } =\|f\|_{L^{p}(\Om )}$. In general, the functional $\|\cdot \|_{p,q}$ is not a norm but a quasi-norm satisfying
\[
\|f+g\|_{p,q} \le C(p,q) \left( \|f\|_{p,q}+\|g\|_{p,q} \right),
\]
where $C(p,q) = \max \{2^{1/p}, 2^{1/p+1/q-1} \}$ (see \cite[Section 1.4.2]{Gr1}). There hold   the following
elementary identities for the quasi-norms $\|\cdot\|_{r}$ and $\|\cdot\|_{p,\infty}$:
\begin{equation}\label{quasi-norm-r}
\int_\Omega |f|^r \, dx = r \int_0^\infty \lambda^{r-1} \mu_f (\lambda )\, d \lambda,
\end{equation}
and
\begin{equation}\label{quasi-norm-pinfty}
\|f\|_{L^{p,\infty}(\Omega )} = \sup_{\lambda \ge 0} \left[ \lambda \mu_f (\lambda )^{1/p} \right]
\end{equation}
for $0< r, p<\infty$, see \cite[Propositions 1.4.5, 1.4.9]{Gr1}. Since
$$
\mu_f (\lambda ) \le \min \{ \|f\|_{L^{p,\infty}(\Omega )}^p \lambda^{-p}, |\Omega | \} \quad\mbox{for}\,\, \lambda >0,
$$
it immediately follows from (\ref{quasi-norm-r}) and (\ref{quasi-norm-pinfty}) that if $\Omega$ has finite measure, then
\begin{equation*}%
 \left( \int_\Omega |f|^r \, dx \right)^{1/r} \le \left( \frac{p}{p-r} \right)^{1/r} |\Omega|^{1/r -1/p } \|f\|_{L^{p,\I}(\Om)}
\end{equation*}
for $0< r< p< \infty$.%

The following is the   H\"{o}lder inequality in Lorentz spaces, essentially due to R. O'Neil \cite{ON}.
\begin{lemma}\label{Lorentz-Holder}
Let $0< p,p_1,p_2< \infty$ and $0< q,q_1,q_2\le\infty$ satisfy
$$
\frac{1}{p}=\frac{1}{p_1}+\frac{1}{p_2}\quad\text{and}\quad \frac{1}{q} \le  \frac{1}{q_1}+\frac{1}{q_2}.
$$
Then there is a constant $C=C(  p_1 , p_2 ,   q_1 , q_2 ,q )>0$ such that
$$
 \|fg\|_{p,q} \le C  \|f\|_{p_1,q_1} \|g\|_{p_2,q_2}
$$
for all $f \in L^{p_1,q_1}(\Omega ) $ and $g \in L^{p_2,q_2}(\Omega )$.
\end{lemma}
\begin{proof} For the case   $1<p<\infty$ and $1 \le q \le  \infty$, the assertion in the lemma was already  proved by R. O'Neil  \cite[Theorem 3.4]{ON}. For  the general case when $0<p<\infty$ and $0 < q \le  \infty$, we recall that
\[
\norm{|h|^r}_{p/r ,q/r} = \norm{h}_{p,q}^r  \quad \mbox{for}\,\, 0<  r < \infty;
\]
see \cite[Section 1.4.2]{Gr1} for example. Hence, if $r $ is chosen so that $0< r<
\min\{p, q\}$, then by   \cite[Theorem 3.4]{ON},
\[
\begin{split}
\|fg\|_{p,q}&= \left\||f|^{r}|g|^{r} \right\|_{p/r, q/r}^{1/r}
\\
& \le C  \left(  \left\||f|^r \right\|_{p_1/r, q_1/r}  \left\| |g|^{r} \right\|_{p_2/r, q_2/r} \right)^{1/r}
\\
&= C  \|f\|_{p_1,q_1} \|g\|_{p_2,q_2}.
\end{split}
\]
The proof of the lemma is completed.
\end{proof}

The  Sobolev inequality can be generalized to Lorentz spaces as follows (see \cite[Remark 7.29]{AdamsFournier}  and \cite{PO}).

\begin{lemma}\label{Lorentz-Sobolev}\label{refined Sobolev} For $1 < p<n$, $1 \le q \le \infty$ or $p=q=1$, there is a constant $C=C(n,p,q)>0$ such that
$$
  \|u \|_{L^{p^* , q}({{\mathbb R}^n} )} \le C  \| \nabla u \|_{L^{p,q}({{\mathbb R}^n} )}
$$
for all $u \in L^{p,q} (\R^n )$ with $\nabla u \in  L^{p,q} (\R^n  ;\R^n)$.
\end{lemma}

If $\Omega$ is bounded, then $\|\cdot \|_{p,\infty}$ is equivalent to the   small scale quasi-norms $\norm{ \cdot}_{p,\infty, (r)}$, defined by
\[
\norm{f}_{p,\infty, (r)}=  \norm{f}_{p,\infty, (r);\Omega }=  \sup_{x \in \Om} \norm{f}_{L^{p,\infty}( \Omega \cap B_r (x))}
\]
for $r>0$.
Here the balls $B_r (x)$ can be replaced by  the   cubes $Q_r (x)$, where $Q_r (x) = x +(-r/2 , r/2)^n$. In fact,
  there is a constant $C=C(n ) >1$ such that
$$
\frac{1}{C}   \norm{f}_{p,\infty, (r)} \le  \sup_{k \in r \mathbb{Z}^n} \norm{f}_{L^{p,\infty}( \Omega \cap Q_r (k))} \le  C \norm{f}_{p,\infty, (r)}.
$$
It should be remarked that
\[
\|f\|_{p,\infty,(r)} \le \|f\|_{p,\infty} \le C \|f\|_{p,\infty,(r)}\quad\mbox{for all}\,\,f \in L^{p,\infty}(\Omega ),
\]
where $C$  depends on $n$ and  $\Omega$ as well as  $r$.

\begin{example}\label{ex2-proof} Here we give details of Example \ref{ex2}.
Let $1 < p<\infty$.
For  $0<r<1$,    we define
\[
f(x) = \sum_{ k \in \mathbb Z^n} \frac {1}{|x-2rk |^{n/p}} \mathbf{1}_{B_r (2rk)} (x)   \quad (x\in \R^n ).
\]
Note that $B_r (2rk) \subset B_1 $ if and only if $|2rk| +r \le  1$, and $B_r (2rk) \cap B_1 \neq \varnothing$ if and only if $|2rk| < r+ 1$.
 Moreover, since the number of $k\in \Z^n$ with $|k|< 1/r$ is approximately equal to $(1/r)^n$ as $r\to 0$, we have
\begin{align*}
\|f\|_{p,\infty; B_1} &\approx  \sup_{\lambda >0} \lambda \left| \bigcup_{|k|\lesssim 1/r} \left\{ x \in B_r (2rk)\,:\, f(x) \ge \lambda \right\} \right|^{1/p} \\
&=  \sup_{\lambda >0} \lambda \left[ \sum_{|k|\lesssim 1/r} \left|\left\{ x \in B_r (2rk)\,:\, |x-2rk|^{-n/p} \ge \lambda \right\}\right| \right]^{1/p} \\
&\approx \sup_{\lambda >0} \lambda \left|\left(\frac{1}{r}\right)^{n} \left[\min \{r , \lambda^{-p/n}\} \right]^n \right|^{1/p} = \left(\frac{1}{r}\right)^{n/p}
\end{align*}
for small $r >0$. Note also that if $x \in B_1$ and $2r|k| \ge 2r+1$, then $B_r (2rk) \cap B_r (x) = \varnothing$. Therefore,
\begin{align*}
\|f\|_{p,\infty, (r); B_1} &\approx
\sup_{|k|\lesssim 1/r} \|f\|_{p,\infty; B_r (2rk)} \\
&\approx   \sup_{\lambda >0} \lambda \left| \left\{ x \in B_r (2rk)\,:\, f(x) \ge \lambda \right\} \right|^{1/p} \approx 1.\hspace{15mm} \square
\end{align*}
\end{example}

To estimate the lower-order  terms in \eqref{bvp} and  \eqref{bvp-dual} in terms of the  quasi-norms $\norm{ \cdot}_{p,\infty, (r)}$, we need the following localized Sobolev inequalities.

\begin{lemma}
\label{localized Sobolev}
The following assertions hold.
\begin{enumerate}
\item[\textup{(i)}] For $1 \le p<n$,   there is a constant $C = C(n,p)>0$   such that for every $x_0 \in \R^n$ and $r>0$, we have
$$
 \|u \|_{L^{p^* , p}( Q_r(x_0)) } \le C  \left( \| \nabla u \|_{L^{p}(Q_r(x_0)) } + \frac{1}{r} \|  u \|_{L^{p}(Q_r(x_0)) } \right)
$$
for all $u \in W^{1,p} ( Q_r(x_0))$.
\item[\textup{(ii)}] For $1 \le q<n/2$,   there is a constant $C = C(n,q)>0$  such that for every $x_0 \in \R^n$ and $r>0$, we have
\[
 \|u \|_{L^{(q^*)^* , q}( Q_r(x_0)) } \le C     \left( \norm{u}_{W^{2,q}(Q_r(x_0))} +\frac{1}{r^2}  \norm{ u}_{L^{q}(Q_r(x_0))} \right)
\]
for all $u \in W^{2,q}( Q_r(x_0))$.
\end{enumerate}
\end{lemma}

\begin{proof} Assume   that $1 \le p< n$ and  $u \in W^{1,p} ( Q_r (x_0 )    )$.  Let $v \in W^{1,p} ( Q_1 (0))$ be defined by    $v (y) = u (x_0 +r y)$
 for   $y \in Q_1 (0)$. Then  $v$ can be easily  extended to $\R^n$  so  that
$\norm{v}_{W^{1,p}(\R^n)} \le C_1 (n, p)  \norm{  v}_{W^{1,p}(Q_1 (0))}$.
Hence by Lemma \ref{refined Sobolev},
$$
 \|v \|_{L^{p^* , p}( Q_1 (0)    ) } \le C_2 (n,p)  \left( \| \nabla v \|_{L^{p}(Q_1 (0)  ) } +   \|  v \|_{L^{p}(Q_1 (0)  ) } \right) .
$$
Note that
\[
\| \nabla v \|_{L^{p}(Q_1 (0)  ) } + \|  v \|_{L^{p}(Q_1 (0)  ) } =r^{1-n/p} \left(  \| \nabla u \|_{L^{p}(Q_r  ) } +  \frac{1}{r}   \|  u \|_{L^{p}(Q_r  ) }  \right),
\]
where $Q_r = Q_r(x_0)$. Moreover, since $\mu_u (\lambda ) = r^{n} \mu_v (\lambda )$ for $\lambda >0$, we have
\begin{align*}
 \|u \|_{L^{p^* , p}( Q_r   ) } &= r^{n/p^*} \|v \|_{L^{p^* , p}( Q_1 (0)    ) } \\
 &\le  C_2 (n,p) \left(  \| \nabla u \|_{L^{p}(Q_r  ) } +  \frac{1}{r}   \|  u \|_{L^{p}(Q_r  ) }  \right),
\end{align*}
and the assertion (i) is proved.

Assume next that $1 \le q< n/2$ and $u \in W^{2,q}( Q_r (x_0 )  )$.
Then  $v$ can be easily  extended to $\R^n$  so  that
$\norm{v}_{W^{2,q}(\R^n)} \le C_3 (n, q)  \norm{  v}_{W^{2,q}(Q_1 (0))}$.
By an elementary interpolation inequality,
\[
\norm{v}_{W^{2,q}(Q_1 (0))} \le C_4 (n, q)  \left( \norm{\nabla^2  v}_{L^{ q}(Q_1 (0))} + \norm{   v}_{L^{ q}(Q_1 (0))} \right).
\]
Hence using Lemma \ref{refined Sobolev} twice, we obtain
\[
\|v \|_{L^{(q^*)^* , q}( Q_1 (0)    ) }      \le C_5 (n, q)  \left( \norm{\nabla^2  v}_{L^{q}(Q_1 (0))} + \norm{   v}_{L^{q}(Q_1 (0))} \right),
\]
from  which  the assertion (ii) follows by exactly the same way as above.
\end{proof}

\subsection{Basic estimates involving weak quasi-norms}

The following  is now  standard and proved in \cite{KangK,KK,KiTs} e.g.

\begin{lemma}
\label{th3-1}
 Let $\Om $ be a bounded Lipschitz domain in $\R^n$ with $n \ge 3$, and let
 $p \in (1,n)$.  Then  there is a constant $C_0 = C_0(  n,  \Omega, p)>0$ such that for every ${\mathbf b} \in L^{n,\infty}(\Om ;\R^n )$, we have
\EQ{\label{th3-1-eq1}
 \norm{u \mathbf{b}}_{p } \le C_0\norm{\bb}_{n,\infty } \norm{u}_{W^{1,p}(\Om)}  \quad\mbox{for all}\,\, u \in W^{1,p}(\Om)
}
and
\EQ{\label{th3-1-eq2}
 \norm{  \mathbf{b} \cdot \nabla v}_{W^{-1,p'} (\Om) } \le C_0 \norm{\bb}_{n,\infty }  \norm{v}_{W^{1,p'}(\Om)}  \quad\mbox{for all}\,\, v \in W^{1,p'}(\Om).
}
In addition, if ${\mathbf b} \in L^{n}(\Om ;\R^n )$, then
for each $\eps>0$ there is a constant $C_\eps = C(\eps , n,   \Omega , p,  \bb)>0$ such that
\EQ{\label{th3-1-eq3}
 \norm{u \mathbf{b}}_{p } \le \eps \norm{\nabla u}_{L^{p}(\Om)} + C_\eps
\norm{u}_{p } \quad\mbox{for all}\,\, u \in W^{1,p}(\Om)
}
and
\EQ{\label{th3-1-eq4}
 \norm{  \mathbf{b} \cdot \nabla v}_{W^{-1,p'} (\Om) } \le \eps \norm{\nabla v}_{L^{p'}(\Om)} + C_\eps
\norm{v}_{p' } \quad\mbox{for all}\,\, v \in W^{1,p'}(\Om).
}
\end{lemma}

Specifically,   \eqref{th3-1-eq1} follows from \cite[Lemma 3.5]{KiTs}, \eqref{th3-1-eq2} is easily deduced from \eqref{th3-1-eq1} by duality, and the estimates \eqref{th3-1-eq3} and \eqref{th3-1-eq4} follow from \cite[Lemmas 3.3, 3.4]{KK}.

\begin{remark}\label{rem-cons-depend}
If ${\mathbf b} \in L^{r}(\Om ;\R^n )$ for some $r \in (n, \infty )$, then the dependence of the constant $C_\eps$  on $\bb$ is only through its $L^r$-norm. Indeed,  for every $u \in W^{1,p}(\Om)$, we have
\begin{align*}
 \norm{u \mathbf{b}}_{p } & \le  \|u\|_{\frac{rp}{r-p}}\|\bb\|_{r}   \le \|u\|_{p}^\theta \|u\|_{p^* }^{1-\theta}\|\bb\|_{r } \\
 & \le \eps \norm{u}_{W^{1,p}(\Om)} + C(\e,n,p, r ,\Omega ) \|\bb\|_{r}^{1 /\theta }\norm{u}_{p},
\end{align*}
where $\theta =1-n/r  >0$. When $\bb \in L^n (\Om ;\R^n )$, the constant $C_\e$ depends on $r>0$ such that $C\norm{b}_{n, \infty, (r)}\le \e$; see the comment after Lemma \ref{th3-2}.
\end{remark}

The following are refined  versions of Lemma \ref{th3-1} in terms of the new quasi-norm $\norm{\bb}_{n,\infty, (r)}$ for $\bb \in L^{n,\infty}(\Om ;\R^n )$. The proofs of \eqref{th3-1-eq3} and \eqref{th3-1-eq4} in \cite[Lemmas 3.3, 3.4]{KK} are based on the possibility  of $C^\infty_c$-approximations of $\bb$ in $L^n(\Om; \R^n )$, which  cannot be directly adapted to prove    \eqref{th3-2eq1}  of Lemma \ref{th3-2} nor \eqref{th3-3eq1} of Lemma \ref{th3-3} below.

\begin{lemma}
\label{th3-2} Let $\Om $ be a bounded Lipschitz domain in $\R^n$ with $n \ge 3$, and let $p \in [1,n)$.
Then there exists a constant $C = C(n,  \Omega, p )>0$   such that for every ${\mathbf b} \in L^{n,\infty}(\Om ;\R^n )$ and $r  \in (0,  \diam \Om )$, we have
\EQ{\label{th3-2eq1}
 \norm{u \mathbf{b}}_{p } \le C  \norm{\bb}_{n,\infty, (r)} \bke{\norm{\nb  u}_{p} +\frac{1}r  \norm{ u}_{p}}
}
for all $u \in W^{1,p}(\Om)$.
\end{lemma}

Note that if ${\mathbf b} \in L^{n}(\Om ;\R^n )$, then for each $\eps>0$ there is $r>0$ such that $\norm{\bb}_{n,\infty, (r)} < \e$. Hence the estimates  \eqref{th3-1-eq1} and \eqref{th3-1-eq3} of Lemma \ref{th3-1} immediately follow  from (\ref{th3-2eq1}).  Lemma \ref{th3-2}  also
shows that  \eqref{th3-1-eq3} holds for $p=1$, which is not stated in \cite[Lemma 3.3]{KK} but   implied by its proof.

\begin{proof} Suppose that  $u \in W^{1,p}(\Om)$. Since $\Om$ is a bounded   Lipschitz domain, it follows from the Stein  extension theorem (see \cite[page 181]{Stein70}) that  $u$ can be extended to $\R^n$ so that
\EQ{\label{extension}
\norm{u}_{W^{1,p}(\R^n)} \le C_1   \norm{  u}_{W^{1,p}(\Om)} \quad\mbox{and}
\quad \norm{u}_{L^{p}(\R^n)}  \le C_1   \norm{  u}_{L^{p}(\Om)}
}
for some $C_1 = C_1 (n,  \Om, p)>0$.
Extend $\bb$ to $\R^n$ by defining $\bb=0$ outside $\Om$.

Let  $k \in \La = r \mathbb{Z}^n$.
Then by  Lemmas \ref{Lorentz-Holder} and  \ref{localized Sobolev}, there is a constant  $C_2 = C_2 (n,p)>0$ such that
\begin{align*}
\norm{u\bb}_{L^p (Q_r (k) )} & \le C_2  \norm{ \mathbf{b}}_{L^{n,\infty} (Q_r (k) )}\norm{u }_{L^{p^* ,p} (Q_r (k) )} \\
& \le C_2 \norm{ \mathbf{b}}_{L^{n,\infty} (Q_r (k) )}\bke{ \norm{\nb u }_{L^p (Q_r (k) )} + \frac 1r \norm{ u }_{L^{p} (Q_r (k) )}} .
\end{align*}
Taking the $p$-th power and summing over $k\in \La$, we have
\EQN{
\norm{u \mathbf{b}}_{L^p (\Om) }^p & \le \sum_{k\in \La} \norm{u \mathbf{b}}_{L^p (Q_r (k))}^p
\\
&\le   \sum_{k\in \La} 2^p C_2^p \norm{\bb}_{n,\infty, (r)}^p
\int_{Q_r (k) } \bke{|\nb u|^p  + \frac{1}{r^{p}}|u|^p} \, dx
\\
& =  2^p C_2^p  \norm{\bb}_{n,\infty, (r)}^p  \bke{ \norm{\nb u }_{L^p (\R^n )} ^p+  r^{-p} \norm{ u }_{L^{p}  (\R^n )}^p}
\\
&\le 2^p C_2^p  \norm{\bb}_{n,\infty, (r)}^p  \left[  C_1^p \left( \norm{\nb u }_{L^p (\Om )} ^p + \norm{ u }_{L^p (\Om )}^p \right) +  r^{-p}  C_1^{p} \norm{ u }_{L^{p}  (\Om )}^p \right].
}
Taking the $p$-th root,
we get \eqref{th3-2eq1} with $C=2 C_1 C_2 \left( 1+\diam \Om \right)$.
\end{proof}

\begin{lemma}
\label{th1-2}
Let $\Om $ be a bounded Lipschitz domain in $\R^n$ with $n \ge 3$.

\begin{enumerate}
\item[\textup{(i)}] For $p \in (n',n)$,   there exists a constant $C = C(n, \Omega , p )>0$  such that for every $c \in L^{n/2,\infty}(\Om )$  and $r  \in (0,  \diam \Om )$, we have
\EQ{\label{th3-6eq1}
   \norm{c u}_{W^{-1,p}(\Om ) } \le C  \norm{c}_{n/2,\infty, (r)} \bke{\norm{\nb  u}_{p} +\frac{1}r  \norm{ u}_{p}}
}
for all $u \in W^{1,p}(\Om)$.
\item[\textup{(ii)}] For $q \in [1,n/2)$,   there exists a constant $C = C(n, \Omega ,q )>0$   such that for every $c \in L^{n/2,\infty}(\Om )$  and $r  \in (0,  \diam \Om )$, we have
\EQ{\label{th3-2eq2}
 \norm{c u}_{q } \le C   \norm{c}_{n/2,\infty, (r)} \bke{\norm{u}_{W^{2,q}(\Om)} +\frac{1}{r^2}  \norm{ u}_{q}}
}
for all $u \in W^{2,q}(\Om)$.
\end{enumerate}
\end{lemma}

\begin{proof} Suppose that $n' <  p<n$ and  $u \in W^{1,p}(\Om)$. As in the proof of Lemma \ref{th3-2}, we extend   $u  $   to $\R^n$  so that it  satisfies the estimates in   \eqref{extension}.
Extend $c$ to $\R^n$ by defining $c=0$ outside $\Om$. Then  by  Lemmas \ref{Lorentz-Holder} and  \ref{localized Sobolev},
\begin{align*}
   \norm{c u}_{L^{np/(n+p), p}(Q_r (k ) ) }& \lec \norm{c}_{L^{n/2,\infty}(Q_r (k ) )} \norm{u}_{L^{p^*,p}(Q_r (k ) )}\\
   &   \lec  \norm{c}_{n/2,\infty, (r)} \left(
   \norm{\nb u}_{L^{p}(Q_r (k )  )} +\frac{1}{r} \norm{u}_{L^{p}(Q_r (k )  )}  \right)
\end{align*}
for each $k \in \La = r \mathbb{Z}^n$. Taking the $p$-th power and summing over $k\in \La$, we obtain
\begin{align*}
   \norm{c u}_{L^{np/(n+p), p}(\Om  )}^p & \le    \sum_{k\in \La} \norm{c u}_{L^{np/(n+p), p}(Q_r (k ) ) }^p \\
   &   \lec  \norm{c}_{n/2,\infty, (r)}^p  \sum_{k\in \La} \left(
   \norm{\nb u}_{L^{p}(Q_r (k) )}^p +\frac{1}{r^p} \norm{u}_{L^{p}(Q_r (k) )}^p  \right)\\
   &   \lec  \norm{c}_{n/2,\infty, (r)}^p   \left(
   \norm{\nb u}_{L^{p}(\Om )}^p +\frac{1}{r^p} \norm{u}_{L^{p}(\Om  )}^p  \right),
\end{align*}
which implies that
\[
\norm{c u}_{\frac{np}{n+p}, p } \lec \norm{c}_{n/2,\infty, (r)}    \left(
   \norm{\nb u}_{p}  +\frac{1}{r} \norm{u}_{p}   \right).
\]
Note that
\[
n'< p' < n \quad \mbox{and}\quad  \frac{n+p}{np}+ \frac{1}{(p')^*} =1 .
\]
Hence for all  $v \in W^{1,p'}_0(\Om)$, we have
\EQN{
 \left|    \int_{\Om} c  uv \, dx \right|
&  \lec \norm{c u}_{\frac{np}{n+p}, p } \norm{v}_{(p')^*,p'}
\\
&\lec  \norm{c}_{n/2,\infty, (r)}    \left(
   \norm{\nb u}_{p}  +\frac{1}{r} \norm{u}_{p}   \right)
  \norm{\nb v}_{p'} ,
}
which completes the proof of (\ref{th3-6eq1}).

Suppose next that $1\leq q < n/2$ and  $u \in W^{2,q}(\Om)$.
By the Stein  extension theorem, $u$ can be extended to $\R^n$ so that
\[
\norm{u}_{W^{2,q}(\R^n)} \le C    \norm{  u}_{W^{2,q}(\Om)} \quad\mbox{and}
\quad \norm{u}_{L^{q}(\R^n)}  \le C    \norm{  u}_{L^{q}(\Om)}
\]
for some $C  =C  (n, q,\Om)$. By  Lemmas \ref{Lorentz-Holder} and  \ref{localized Sobolev}
\begin{align*}
   \norm{c u}_{L^q (Q_r (k ) ) }& \lec \norm{c}_{L^{n/2,\infty}(Q_r (k) )} \norm{u}_{L^{(q^* )^*, q}(Q_r (k ) )}\\
   &   \lec  \norm{c}_{n/2,\infty, (r)} \left(
   \norm{\nb^2 u}_{L^{q}(Q_r (k) )} +\frac{1}{r^2} \norm{u}_{L^{q}(Q_r (k) )}  \right).
\end{align*}
for each $k \in    r \mathbb{Z}^n$. Hence  taking the $p$-th power and summing over $k\in   r \mathbb{Z}^n$, we can complete the proof of (\ref{th3-2eq2}).
\end{proof}

\begin{lemma}
\label{th3-3} Let $\Om $ be a bounded Lipschitz domain in $\R^n$ with $n \ge 3$, and let $p \in (n',n)$. Then there exists a constant $C = C(n,  \Omega,p)>0$   such that for every ${\mathbf b} \in L^{n,\infty}(\Om ;\R^n )$ with  $\div \bb \in L^{n/2,\infty}(\Om  )$ and $r  \in (0,  \diam \Om )$, we have
\EQ{\label{th3-3eq1}
\norm{  \mathbf{\bb} \cdot \nabla v}_{W^{-1,p } (\Om) } \le C M_r (\bb) \bke{  \norm{\nb v}_{p } + \frac{1}{r} \norm{v}_{p }}
}
for all $v \in W^{1,p }(\Om)$,  where  $M_r (\bb) = \norm{\bb}_{n,\infty,(r)} + \norm{\div \bb}_{n/2,\infty,(r)}$.
\end{lemma}

\begin{remark} Compared with \eqref{th3-1-eq2} and \eqref{th3-1-eq4} of Lemma \ref{th3-1}, the estimate (\ref{th3-3eq1}) holds  for a more restricted range of $p$: $n'<p<n$ instead of  $1 < p <n$
See also Lemma \ref{app-corr2} below for another estimate of similar type.
\end{remark}

\begin{proof} Suppose that  $v\in W^{1,p}(\Om)$. Then since $\bb$ has the weak divergence in $L^{n/2,\infty}(\Om  )$, it follows that
\[
 \div \,(v \bb ) =  \mathbf{\bb} \cdot \nabla v +  ( \div \bb ) v .
\]
By Lemma \ref{th3-2},
\[
 \norm{ \div (v \bb )}_{W^{-1,p } (\Om)} \le  \norm{v \mathbf{b}}_{p } \le C  \norm{\bb}_{n,\infty, (r)} \bke{\norm{\nb  v}_{p} +\frac{1}r  \norm{ u}_{p}},
\]
while by  Lemma \ref{th1-2} (i),
\[
\norm{( \div \bb ) v }_{W^{-1,p}(\Om ) } \le C  \norm{\div \bb}_{n/2,\infty, (r)} \bke{\norm{\nb  v}_{p} +\frac{1}r  \norm{ v}_{p}}.
\]
Combining these two estimates, we complete the proof of the lemma.
\end{proof}

\subsection{Mollification of functions in Lorentz spaces}

The following lemma is  the Young-O'Neil convolution inequality in Lorentz spaces on $\R^n$.
\begin{lemma}\label{Young-ONeil} The following assertions hold.
\begin{enumerate}
\item[\textup{(i)}] Let $1 < p < \infty$ and $1\le q \le\infty$. Then there is a constant $C=C(  p )>0$ such that
$$
 \|f*g\|_{p,q} \le C  \|f\|_{p ,q } \|g\|_{1} ,
$$
for all $f \in L^{p,q}(\R^n) $ and $g \in L^{1}(\R^n )$.
\item[\textup{(ii)}] Let $1 <p,  p_1,p_2< \infty$ and $1\le q, q_1,q_2\le\infty$ satisfy
$$
\frac{1}{p}+1=\frac{1}{p_1}+\frac{1}{p_2}\quad\text{and}\quad \frac{1}{q} \le \frac{1}{q_1}+\frac{1}{q_2}.
$$
Then there is a constant $C=C(  p_1 , p_2 ,   q_1 , q_2, q )>0$ such that%
$$
 \|f*g\|_{p,q} \le C  \|f\|_{p_1,q_1} \|g\|_{p_2,q_2} ,
$$
for all $f \in L^{p_1,q_1}(\R^n) $ and $g \in L^{p_2,q_2}(\R^n )$.
\end{enumerate}
\end{lemma}

We remark that Lemma \ref{Young-ONeil}   (i)  is an immediate consequence of the real interpolation result
\[
L^{p,q}(\R^n ) =  \left(L^1 (\R^n ) , L^\infty (\R^n ) \right)_{1-1/p, q};
\]
see \cite[Remark 7.29]{AdamsFournier} for more details. Lemma \ref{Young-ONeil} (ii)  was proved by O'Neil \cite[Theorem 2.6]{ON}  and clarified by Yap \cite{Yap}. See also Blozinski \cite{Bl} for some counterexamples to the  endpoint case  $p=\infty$ or $p_1=1$.

We now prove several results for mollifications of functions in Lorentz spaces.
Let $ \Phi \in C^\infty_c(\R^n)$ be a fixed  non-negative
 function  with $\int_{\R^n}\Phi(x) \, dx = 1$. For $\rho>0$, we define
$\Phi_\rho(x) = \rho^{-n}\Phi ( x/\rho )$ for all $x \in \R^n$.
Then since $\int_{\R^n} \Phi_\rho (x) \, dx =1$ for any $\rho>0$, it follows from Lemma \ref{Young-ONeil} (i) that if $1<p<\infty$ and $1 \le q\le \infty$, then
\begin{equation}\label{kernel-lemma}
\|f *\Phi_\rho  \|_{p, q} \leq C(p) \|f\|_{p, q}  \quad\mbox{for all}\,\,  f \in L^{p, q}(\R^n).
\end{equation}
\begin{lemma} \label{app-corr} Let  $f \in L^{p, q}(\R^n )$ with $1 < p <  \infty $ and $ 1 \le  q <  \infty$. Then for every $\e >0$, there exists  $\rho_0 >0$  such that
\[
\sup_{0<\rho\le \rho_0}\norm{f- f*\Phi_\rho}_{p, q}  \le \e .
\]
\end{lemma}
\begin{proof} Let $\e>0$ be given. Since $q$ is finite, it follows from \cite[Theorem 1.4.13]{Gr1} that  %
the set of finitely simple functions is dense in $L^{p,q}(\R^n ) $.  Hence  there is a   simple function
$f_1 = \tsum_{j=1}^N   c_j \mathbf{1}_{E_j}$,  where the sets $E_j$ have finite measure and are pairwise disjoint,
such that $f_2 = f - f_1$ satisfies $\norm{f_2}_{p,q} \le \e$.   By (\ref{kernel-lemma}), we have
\[
\begin{split}
\norm{f- f*\Phi_\rho}_{p,q } & \leq  C_0 \norm{f_{1} - f_1*\Phi_\rho}_{p,q}+ C_0 \norm{f_{2} }_{p,q}+ C_0 \norm{f_2*\Phi_\rho}_{p,q}\\
& \leq C_0 \norm{f_{1} - f_1*\Phi_\rho}_{p,q} + C_0 (1+ C_1 ) \norm{f_{2} }_{p,q}\\
& \leq C_0 \norm{f_{1} - f_1*\Phi_\rho}_{p,q} +C_0 (1+ C_1 ) \e ,
\end{split}
\]
where $C_i = C_i (p,q)$ for $i=0,1$.
  Since $f_1$ is a finitely simple function, there exists $\rho_0 >0$ such that
\[
\sup_{0 < \rho \le \rho_0} \norm{f_{1} - f_1 *\Phi_\rho}_{p,q} \leq \e
\]
and therefore
\[
\sup_{0 <\rho \le \rho_0} \norm{f- f *\Phi_\rho}_{p,q }    \leq     C_0 (2 +C_1 )  \e  .
\]
The assertion is proved since $\e>0$ is arbitrary.
\end{proof}

Next, we introduce the following lemma which will be  used in the proof of  Proposition \ref{prop4.4} below that proves Theorem \ref{th4-q version} when $n'<p<2$.
The assertion of the lemma is in the same spirit as those of \eqref{th3-1-eq2} and \eqref{th3-1-eq4} of Lemma \ref{th3-1},  and \eqref{th3-3eq1} of Lemma \ref{th3-3}.
\begin{lemma}
\label{app-corr2} Let $\Om $ be a bounded Lipschitz domain in $\R^n$ with $n \ge 3$, and let $p \in (n',n)$. Then for each  $\e >0$ and $  \delta\in (0,1)$, there exists a constant $C_{\e, \delta}  = C(n,  \Omega , p,  \Phi,   \e , \delta)>0$  such that for every ${\mathbf b} \in L^{n,\infty}(\Om ;\R^n )$ and $v \in W^{1,p }(\Om)$, we have
\[
\norm{ ( \mathbf{\bb} *\Phi_\rho )\cdot \nabla v}_{W^{-1,p } (\Om) } \le  \|\bb \|_{n,\infty} \bke{ \e  \norm{\nb v}_{p } + \frac{C_{\e,\delta}}{\rho^{1+\delta}} \norm{v}_{p }}
\]
for all  $  \rho \in (0,1)$, where    $\bb$ is extended to $\R^n$ by  defining as zero outside $\Omega$.
\end{lemma}

\begin{proof} Writing   $\mathbf{\bb}^\rho = \mathbf{\bb} *\Phi_\rho $,
we have to estimate
\[
  \int_\Om   u (  \bb^\rho  \cdot \nabla v )  \, dx  = - \int_\Om \nabla u \cdot (\bb^\rho  v ) \, dx
-  \int_\Om  u  (\div \bb^\rho)  v\, dx
\]
for any $u \in W^{1,p'}_0(\Om)$ with $\norm{u}_{W^{1,p'}_0(\Om)}=1$.

One naive estimate would be
\EQN{
  &
 \abs{\int_\Om \nabla u \cdot ( \bb^\rho  v ) \, dx
+  \int_\Om  u  (\div \bb^\rho)  v\, dx }
\\
& \qquad \le \norm{\nb u}_{p'} \norm{\bb^\rho}_\infty \norm{ v}_{p} + \norm{ u}_{(p')^*} \norm{\div \bb^\rho}_n \norm{ v}_{p},
}
and try to bound by $\norm{\bb}_{n,q}$ for some $q>n$ as $\bb$ may not be in $L^n$ as follows:
\begin{align*}
 \norm{\bb^\rho}_\infty &\le  C \norm{\Phi_\rho}_{n' ,q'}\norm{\bb}_{n,q} =   \frac{C}{\rho}\norm{\Phi}_{n' ,q'}\norm{\bb}_{n,q};\\
 \norm{\div \bb^\rho}_n &\le  C \norm{\nb \Phi_\rho}_1 \norm{\bb}_{n,q}
 = \frac {C}{\rho} \norm{\nb \Phi}_1 \norm{\bb}_{n,q}.
\end{align*}
This idea unfortunately fails because Lemma \ref{Young-ONeil} is invalid for $p=\infty$ or
 $p_1= 1$ when $q>n$ (see \cite{Bl}).

We modify the above estimate with slightly different exponents.
Let $l$ and $s$ be defined by
\[
 \frac{1}{l} =  \frac{1}{p} -  \frac{\delta}{n}  \quad \text{and} \quad  \frac{1}{s} = \frac{1}{p}+\frac{\delta}{n} .
\]
Since $n' < p < n$ and $0< \delta < 1$, we have
$$
1< s< p < l <  p^*  = \frac{np}{n-p} < \infty.
$$
By H\"{o}lder's inequality,
\begin{align} \notag
\abs{ \int_\Om   u (  \bb^\rho  \cdot \nabla v )  \, dx}  & =
 \abs{\int_\Om \nabla u \cdot (\bb^\rho  v ) \, dx
+  \int_\Om  u  (\div \bb^\rho )  v\, dx }
\\ \label{n-112623.est}
&   \le \norm{\nb u}_{p'} \norm{\bb^\rho}_{l_1} \norm{ v}_{l} + \norm{ u}_{(p')^*} \norm{\div \bb^\rho}_{s_1} \norm{ v}_{s},
\end{align}
where   $l_1  = n/\delta  $  and $s_1 = n/(1-\delta )$, so that
\[
  \frac{1}{l_1} + \frac{1}{l} = \frac{1}{p} \quad\mbox{and}\quad   \frac{1}{s_1} + \frac{1}{s} =\frac{1}{p} + \frac{1}{n} .
\]
Let $l_2 $ and $s_2$ be given by
\[
\frac{1}{l_1} +1 = \frac{1}{l_2} + \frac{1}n \quad\text{and} \quad \frac{1}{s_1} +1 = \frac{1}{s_2} + \frac{1}{n}.
\]
Note then that
$$
  n< l_1 , s_1 < \infty, \quad 1< l_2 < l_1  , \quad\mbox{and}\quad 1<s_2 < s_1  .
$$
Hence by the Young-O'Neil convolution inequality  (Lemma \ref{Young-ONeil}),
\[
 \norm{\bb^\rho}_{l_1}  \le  C \norm{\Phi_\rho}_{l_2}\norm{\bb}_{n,\infty} =  C\rho^{-n+\frac{n}{l_2}} \norm{\Phi}_{{l_2}}\norm{\bb}_{n,\infty}
\]
and
\[
 \norm{\div \bb^\rho}_{s_1} \le  C \norm{\nb \Phi_\rho}_{s_2}\norm{\bb}_{n,\infty}
 =C\rho^{-n-1+\frac{n}{s_2}} \norm{\nb \Phi}_{s_2}\norm{\bb}_{n,\infty}.
\]
From these estimates, \eqref{n-112623.est}, and since $\norm{u}_{W^{1,p'}_0(\Om)}=1$, it follows that
\EQN{
&\left| \int_\Om   u (  \bb^\rho  \cdot \nabla v )  \, dx \right|\\
&\qquad \le C \norm{\bb}_{n,\infty} \bke{\norm{\Phi}_{l_2}\rho^{-n+\frac{n}{l_2}}  \norm{ v}_{l} +  \norm{\nb \Phi}_{s_2} \rho^{-n-1+\frac{n}{s_2}}  \norm{ v}_{s}} \\
& \qquad \le  C_{\Phi} \norm{\bb}_{n,\infty} \bke{ \rho^{-n+\frac{n}{l_2}}  \norm{ v}_{l} +  \rho^{-n-1+\frac{n}{s_2}}  \norm{ v}_{s}} ,
}
where $C_{\Phi}  =  C (n, p,\delta  ) \bke{\norm{\Phi}_{l_2} +  \norm{\nb \Phi}_{s_2}}  $.
Observe  that
\[
  n-\frac{n}{l_2} = 1-\frac n{l_1} = 1  - \delta
\quad\mbox{and}\quad
\frac{1}{l} =\frac{\delta}{p^*}+ \frac{1 -\delta }{p}.
\]
Hence by the interpolation inequality in $L^l$, the Sobolev inequality, and  Young's inequality,
\EQN{%
\rho^{-n+\frac{n}{l_2}}\norm{ v}_{l}& \le  \rho^{-1+\delta}  \norm{  v}_{p^*}^{\delta} \norm{v}_{p}^{1-\delta}  \\
 &\le C \left(    \norm{\nb v}_{p} +\|v\|_{p} \right)^{\delta} \left(\rho^{-1}\norm{v}_{p} \right)^{1-\delta}  %
\\
&\le \eta \norm{\nb v}_{p} + \frac {C(\eta)} \rho \norm{v}_{p}
}
for any $\eta>0$.
Observe  also that
$$
  n- \frac{n}{s_2}   =   \frac{n}{s}-\frac{n}{p} =\delta \quad\mbox{and}\quad\norm{ v}_{s}\le |\Om|^{1/s -1/p} \norm{ v}_{p}
.
$$
Therefore, for any $\eta >0$, we have
\[
\left| \int_\Om   u (  \bb^\rho  \cdot \nabla v )  \, dx \right|
 \le C_{\Phi} \norm{\bb}_{n,\infty} \left[\eta \norm{\nb v}_{p} + \frac { C(\eta) \rho^\delta +1}{\rho^{1+\de}  } \norm{v}_{p} \right].
\]
Taking $\eta = \e /C_{\Phi}$, we complete the proof of the lemma.
\end{proof}

\subsection{Miranda-Nirenberg interpolation inequalities}

We shall make crucial use of the following estimate, which is a special case of the Miranda-Nirenberg interpolation inequalities \cite{Miranda, Nirenberg}.   \begin{lemma} \label{MNineq}
Let $\Om $ be a bounded Lipschitz domain in $\R^n$, and let $p \in [1,n)$, $\alpha  \in (0,1)$, and $r = (2-\al)p /(1-\al )$. Then there exists a positive constant $C=C(n, \Om, p , \alpha  ) $ such that
\begin{equation*}%
  \norm{\nb u}_{L^r(\Om)} \le C \left(%
 \norm{ u}_{W^{2,p}(\Om)}+\norm{ u}_{C^\al(\overline \Om)} \right)
\end{equation*}
for all $u \in W^{2,p}(\Omega )\cap C^\alpha (\overline \Om)$.
\end{lemma}

\section{Proof  of Theorem  \ref{th4-q version}  }
\label{Sec4}

We begin with the following   a priori estimates, which can be derived by taking $\phi = u/(1+|u|)$ as a test function for \eqref{bvp} as in \cite{Dr,KK}.

\begin{lemma}\label{log-estimate} Let $\Om $ be a bounded Lipschitz domain in $\R^n$ with $n \ge 3$. Suppose that $\bb \in L^{n,\infty}(\Om; \R^n )$, $ c \in L^{n/2,\infty}(\Om ) $, and $c \ge 0$ in $\Om$. Then there exists  a positive constant $C =C(n,\Om ) $ such that if $u \in W_0^{1,2}(\Om )$ is a weak solution of \eqref{bvp} with $f \in W^{-1,2}(\Om )$, then
\[
\| \ln (1+|u|)\|_{W^{1,2}(\Om )} \le C \left( \|\bb\|_{L^2 (\Om )}+ \|f\|_{W^{-1,2}(\Om )} \right)
\]
and
\[
\left| \{x \in \Om \,:\, |u(x)| \ge k \} \right| \le \frac{C \left( \|\bb\|_{L^2 (\Om )}+ \|f\|_{W^{-1,2}(\Om )} \right)^2 }{[\ln (1+k)]^2}
\]
for all $k >0$.
\end{lemma}
\begin{proof} We sketch  the proof for the sake of completeness.
Let $u  $ be a weak solution of \eqref{bvp} with $f \in W^{-1,2}(\Om )$.
Then  taking  $\phi = u / (1+|u|)$ in the weak formulation (\ref{eq1.3}) for \eqref{bvp}, we obtain
\begin{equation*} %
\int_{\Omega} \left[ \frac{|\nabla u|^2}{(1+|u|)^2} + \frac{c u^2}{1+|u|} \right]  dx = \int_{\Omega} \frac{(\bb\cdot \nabla u) u }{(1+|u|)^2} \,  dx + \left\langle f, \frac{u}{1+|u|} \right\rangle .
\end{equation*}
By the nonnegativity of   $c$, H\"{o}lder's inequality, and   Young's inequality,
\[
\int_{\Omega} \frac{|\nabla u|^2}{(1+|u|)^2} \, dx  \leq \frac{1}{2} \int_{\Omega} \frac{|\nabla u|^2}{(1+|u|)^2}\,  dx + 4 \left(\|\bb\|_{L^2(\Omega)}^2  + \|f\|_{W^{-1,2}(\Om)}^2\right).
\]
From this and the Poincar\'{e} inequality, we see that
\[
\|\ln(1+|u|) \|_{W^{1,2}_0(\Om)} \leq C(n,\Omega) \left(\|\bb\|_{L^2(\Omega)} + \|f\|_{W^{-1,2}(\Om)} \right).
\]
Next, applying  Chebyshev's inequality, we obtain
\[
\begin{split}
  |\{x \in \Om: |u(x)| \ge k\}| & = |\{x \in \Om: \ln (1 +|u(x)|) \ge  \ln(1+k) \}| \\
& \leq \frac{1}{[\ln(1+k)]^2} \int_{\Omega} |\ln(1+|u |)|^2 \,  dx \\
&\leq \frac{C(n,\Omega)}{[\ln(1+k)]^2} \left( \|\bb\|_{L^2(\Omega)}  + \|f\|_{W^{-1,2}(\Om)}\right)^2 .
\end{split}
\]
This  completes the proof of the lemma.
\end{proof}

The following is a key a priori estimate for the proof of   Theorem \ref{th4-q version}.
\begin{lemma}\label{key a priori estimate} Let $\Om $ be a bounded $C^1$-domain in $\R^n$ with $n \ge 3$, and let  $p \in [2, n)$ and $M \in (0,\infty)$. Then  there is a small number $\varepsilon_0>0$, depending  only on $n,  \Omega, p$, and $M$, such that
the following statement holds:

Assume  that  $\bb = \bb_1 + \bb_2$,   $ (\mathbf{b}_1 , \mathbf{b}_2 )   \in L^{n,\I}(\Om;\R^{2n})$,   $  c \in L^{n/2 ,\infty}(\Om )$, $\|\bb_1\|_{n, \infty}   + \norm{c}_{n/2,\infty} \le M$, and  $\div \mathbf{b}_1 \ge 0 , \, c \ge 0$ in $\Om$. Assume also that   $\bb_2$ satisfies
$$
\|\bb_2\|_{n, \I, (r)} \leq \varepsilon_0
$$
for some $r  \in (0,  \diam \Om )$.

Then there exists a positive constant $C $  depending only on $n,   \Omega , p,r, M$, and $\|{\mathbf b}\|_{2}$ such that
\begin{equation} \label{0309.est}
 \| u \|_{W^{1,p} (\Om )} \le C   \|-\De u + \lambda \, \div  ( u \mathbf{b}   ) + \lambda cu \|_{W^{-1,p}(\Om )}
\end{equation}
 for all $u \in W_0^{1,p}(\Om )$  and $\lambda \in [0,1]$.
\end{lemma}

\begin{proof}  Given $u \in W_0^{1,p}(\Om )$ and $\lambda \in [0,1]$, let $f=  -\De u + \lambda \, \div   ( u \mathbf{b} ) + \lambda cu $.  By the assumptions and Lemmas \ref{th3-2} and \ref{th1-2}, we see that $f \in W^{-1,p}(\Om)$. By a simple scaling argument,
we only need to prove \eqref{0309.est} under  the assumption that
\begin{equation} \label{f-scaled}
\|f\|_{_{W^{-1,p}(\Om )}} \leq 1.
\end{equation}
Observe that
\[
-\De u + \lambda \, \div( u \mathbf{b}_1 ) + \lambda cu =f -\lambda \, \div( u \mathbf{b}_2 ) \quad\mbox{in}\,\,\Omega .
\]
Then  it follows from   \cite[Theorem 2.1]{KiTs} that
\begin{align*}
 \| u \|_{W^{1,p} (\Om )} & \le C_1   \|f -  \lambda \, \div( u \mathbf{b}_2 )\|_{W^{-1,p}(\Om )}\\
 & \le C_1 \left( \|f\|_{_{W^{-1,p}(\Om )}}   +  \|  u \mathbf{b}_2\|_{L^p (\Om )}\right) ,
\end{align*}
where $C_1 = C_1(n,  \Om, p,  M )>0$. Moreover,  by \eqref{f-scaled} and Lemma \ref{th3-2}, there exists    $C_2 = C_2 (n,  \Omega, p)>$ such  that
\[
 \| u \|_{W^{1,p} (\Om )}   \le C_1      + C_1 C_2 \|\bb_2\|_{n,\I, (r)}  \left( \| u \|_{W^{1,p} (\Om )}+ \frac{1}{r}\|u \|_{p} \right)
\]
for any $r  \in (0,  \diam \Om )$. Therefore,  assuming  that
$$
\|\bb_2\|_{n, \I, (r)} \leq \varepsilon_0 = \frac{1}{2C_1 C_2}
$$
for some $r  \in (0,  \diam \Om )$, we obtain
\begin{equation}  \label{u.est-0308}
 \| u \|_{W^{1,p} (\Om )}   \le 2C_1 + \frac{1}{r}     \|u \|_{p} .
\end{equation}

We next remove the term $\|u \|_{p}$ in the right hand side of (\ref{u.est-0308}). For $k>0$, let $A_k = \{x \in \Om: |u(x)| >k\}$.
Then since $p \geq 2$ and $\Omega$ is bounded, it follows from \eqref{f-scaled} and Lemma \ref{log-estimate} that
\[
\left| A_k \right| \le \frac{C_3 }{[\ln (1+k)]^2} \quad\mbox{for all}\,\, k>0 ,
\]
where $C_3=C_3(n,   \Om , p, \|\bb\|_{2})>0$. Hence by the H\"{o}lder and Sobolev inequalities, %
\begin{align*}
\| u \|_{W^{1,p} (\Om )} & \le 2C_1  + \frac{1}{r} \left(  \|  u  \|_{L^p (A_k )}+  \|  u  \|_{L^p (\Om \setminus A_k )} \right) \\
& \le 2C_1    + \frac{1}{r} \left( |A_k |^{1/p -1/p^*} \|  u  \|_{L^{p^*} (A_k )}+ |\Omega|^{1/p} k \right)  \\
&\le 2C_1 + \frac{C_4}{r} \left( \left[\frac{1}{{\ln (1+k)}}\right]^{2/n} \| u \|_{W^{1,p} (\Om )}+  k   \right),
\end{align*}
where $C_4 =   C_4(n,   \Om ,p, \|\bb\|_{2})>0$. Therefore,  choosing $k$ sufficiently large so that
$$
\frac{C_4}{r}   \left[\frac{1}{{\ln (1+k)}}\right]^{2/n} < \frac 12 ,
$$
 we find
\[
\| u \|_{W^{1,p} (\Om )} \leq C \left( n,   \Om, p,r, M , \|\bb\|_{2}  \right) .
\]
The proof is then completed.
\end{proof}

The following proposition is just the case $2\le p<n$ of Theorem  \ref{th4-q version}.

\begin{proposition}\label{consequence-key a priori estimate}
Let $\Om$ be a bounded  $C^1$-domain in $\R^n$ with $n \ge 3$, and let $p \in [2,n)$ and $M \in (0,\infty)$.
Assume  that  $\bb = \bb_1 + \bb_2$,   $ (\mathbf{b}_1 , \mathbf{b}_2 )   \in L^{n,\I}(\Om;\R^{2n})$,   $  c \in L^{n/2 ,\infty}(\Om )$, $\|\bb_1\|_{n, \infty}   + \norm{c}_{n/2,\infty} \le M$, and  $\div \mathbf{b}_1 \ge 0 , \, c \ge 0$ in $\Om$.
Assume also  that  $\bb_2$ satisfies
\[
\|\bb_2\|_{n, \I, (r)} \leq \varepsilon_0
\]
for some $r  \in (0,  \diam \Om )$, where $\eps_0$ is the same number as in Lemma \ref{key a priori estimate}. 
\begin{enumerate}
\item[\textup{(i)}] For each $f \in W^{-1,p}(\Om  )$,
 there exists a unique  $p$-weak solution  $u  $ of \eqref{bvp}. Moreover, we have
$$
 \| u \|_{W^{1,p} (\Om )} \le C   \|f\|_{W^{-1,p}(\Om )}.
$$
\item[\textup{(ii)}] For each $g \in W^{-1,p'}(\Om )$,
there exists a unique  $p'$-weak solution   $v  $ of \eqref{bvp-dual}. Moreover, we have
$$
 \|  v \|_{W^{1,p'} (\Om )} \le C   \|g\|_{W^{-1, p'}(\Om )}.
$$
\end{enumerate}
Here the  constant $C>0$  depends  only on $n,   \Omega , p,r,  M$, and $\|{\mathbf b}\|_{2}$.
\end{proposition}

\begin{proof}  Part (i)   follows from Lemma \ref{key a priori estimate} by the method of continuity. Indeed, if $L_0$ and $L_1$ are bounded linear operators from $W_0^{1,p}(\Om )$ into $W^{-1,p}(\Om )$ defined by
\[
L_0 u = -\Delta u \quad\mbox{and}\quad L_1 u = -\De u + \div \,( u \mathbf{b} ) + c u ,
 \]
then by Lemma \ref{key a priori estimate}, we have
\[
\|u\|_{W_0^{1,p}(\Om )} \le C \| (1-\lambda) L_0 u+ \lambda L_1 u \|_{W^{-1,p}(\Om )}
\]
for all $u \in W_0^{1,p}(\Om )$ and $\lambda  \in [0,1]$, which implies that   $L_1$ is bijective. This proves Part (i). Then Part (ii)  follows from Part (i) by a simple duality argument (see the proof of  \cite[Proposition 6.1 (ii)]{KiTs} e.g.).
\end{proof}

The case $n'<p<2$ of Theorem  \ref{th4-q version} is implied by the following more general result.

\begin{proposition}
\label{prop4.4}
Let $\Om$ be a bounded  $C^1$-domain in $\R^n$ with $n \ge 3$, and let $p \in (n',2)$ and $M \in (0,\infty)$. Then  there is a small number $\varepsilon_1 >0 $, depending only on $n,  \Omega, p$,  and  $M $,  such that the following assertions hold:

Assume that  $\bb = \bb_1 + \bb_2$,   $ (\mathbf{b}_1 , \mathbf{b}_2 )  \in L^{n,\I}(\Om;\R^{2n})$, $\bb_2 =\bb_{21} + \bb_{22}$,    $\mathbf{b}_{21} \in L^{n,q}(\Om;\R^n)$ for some $1\le q<\infty$,  $ (\div \bb_1, {\rm div}\, {\mathbf b_{22}} , c )\in L^{n/2 ,\infty}(\Om ; \R^3 )$, $\|\bb_1\|_{n, \infty}   + \norm{(\div \bb_1 , c ) }_{n/2,\infty} \le M$, and $\div \mathbf{b}_1 \ge 0 , \, c \ge 0$ in $\Om$. Assume also  that     $\bb_2$ satisfies
\EQ{\label{eq4.5}
\norm{\bb_{21}}_{n,\infty, (r)} + \norm{\bb_{22}}_{n,\infty, (r)}   +\norm{\textup{div} \,\bb_{22}}_{n/2,\infty, (r)}    \le \e_1
}
for some $r  \in (0,  \diam \Om )$.
\begin{itemize}
\item[\textup{(i)}] For each $f \in W^{-1,p}(\Om  )$, there exists a unique $p$-weak solution $u $ of \eqref{bvp}. Moreover, we have
$$
 \| u \|_{W^{1,p} (\Om )} \le C   \|f\|_{W^{-1,p}(\Om )}.
$$
\item[\textup{(ii)}] For each $g \in W^{-1,p'}(\Om )$, there exists a unique $p'$-weak solution $v $ of \eqref{bvp-dual}. Moreover, we have
$$
 \|  v \|_{W^{1,p'} (\Om )} \le C   \|g\|_{W^{-1,p'}(\Om )}.
$$
\end{itemize}
 Here the constant  $C>0$ depends only on $n,    \Omega , p,q,r$,     $M$,  $ \| {\mathbf b}\|_{2}$,  and   $\bb_{21}$.
\end{proposition}
We remark that as $\bb_{21} \in L^{n,q}(\Omega)$ the condition in \eqref{eq4.5} imposed on $\bb_{21}$ holds for sufficiently small $r$. However, we include it to explicitly specify the choice of $r$.

\begin{proof}
Once Part (ii) is proved, Part (i)  follows by a duality argument. Hence it suffices  to prove (ii). By the method of continuity as in the proof of Proposition \ref{consequence-key a priori estimate}, it suffices to prove that there is a positive constant $C $ depending only on $n, p,  q, r,  \Omega$, $M$,  and   $\bb_{21}$    such that
\begin{equation} \label{0309.v-est}
\|v\|_{W^{1,p'}(\Om)} \leq C \| -\Delta v -\lambda ( \mathbf{b} \cdot \nabla v ) + \lambda cv \|_{W^{-1, p'}(\Om)}
\end{equation}
for every $v \in W_0^{1,p'}(\Om )$  and $\lambda \in [0,1]$, provided that the smallness condition \eqref{eq4.5} is satisfied.

Given $v \in W_0^{1,p'}(\Om )$  and $\lambda \in [0,1]$, we define $g = -\De v - \lambda ( \mathbf{b} \cdot \nabla v ) +\lambda cv$. Then it follows from Lemmas \ref{th1-2} and \ref{th3-3} that  $g \in W^{-1,p'}(\Om )$. By
a scaling argument,
we may assume that
$$%
\|g\|_{W^{-1, p'}(\Om)} \leq 1.
$$%
Note   that  $v   \in W_0^{1,p'}(\Om )$  satisfies
\[
-\Delta v -\lambda \mathbf{b_1} \cdot \nabla v +\lambda c v =  g + \lambda  \mathbf{b_2} \cdot \nabla v \quad\mbox{in}\,\,\Om .
\]
Hence by \cite[Theorem 2.1]{KiTs}, there exists
$C_2=C(n,\Om, p, M)>0$ such that
\begin{align} \notag
\|v\|_{W^{1,p'}(\Om)} & \leq C_2 \Big( \| g\|_{W^{-1, p'}(\Om)} + \|\bb_{2} \cdot \nabla v \|_{W^{-1, p'}(\Om)} \Big) \\ \label{0423a}
& \leq C_2 + C_2 \|\bb_{2} \cdot \nabla v \|_{W^{-1, p'}(\Om)}.
\end{align}
Recall the decomposition $\bb_2 = \bb_{21}+\bb_{22}$.
Then by Lemma \ref{th3-3},
\[
\|\bb_{22} \cdot \nabla v \|_{W^{-1, p'}(\Om)}
\le  C_3M_r (\bb_{22} )  \left( \|\nabla v
\|_{p'} + \frac{1}{r} \|v\|_{p'} \right)
\]
for some $C_3=C_3(n,\Om, p)>0$, where
$$
M_r (\bb_{22} ) = \|\bb_{22}\|_{n, \infty, (r)} + \|\text{div}\, \bb_{22} \|_{n/2, \infty, (r)}.
$$
Define
\begin{equation} \label{epsilon-1.def}
\varepsilon_1: = \min \left\{ \frac{1}{2}\varepsilon_0 ,  \frac{1}{4C_2C_3} \right\} ,
\end{equation}
where $\varepsilon_0$ is the same constant as in Proposition \ref{consequence-key a priori estimate} with $p=2$. Then by the smallness condition  (\ref{eq4.5}),  we obtain
\EQ{\label{0423b}
C_2 \|\bb_{22} \cdot \nabla v \|_{W^{-1, p'}(\Om)}
\le  \frac14  \left( \|\nabla v
\|_{p'} + \frac{1}{r} \|v\|_{p'} \right).
}

To estimate $\|\bb_{21} \cdot \nabla v \|_{W^{-1, p'}(\Om)}$, we fix some nonnegative $ \Phi \in C^\infty_c(\R^n)$ with $\int_{\R^n} \Phi(x) \,dx = 1$  and define  $\Phi_\rho(x) = \rho^{-n}\Phi(x/\rho)$ on $\R^n$ for $\rho >0$.
 Let $C_0$ be the constant in Lemma \ref{th3-1}.
Then since $\bb_{21}\in L^{n,q}(\Om)$ and $q<\infty$, it follows from Lemma \ref{app-corr} that
there is     $\rho  = \rho (\bb_{21} ) >0 $   such that
\EQ{\label{r.cond2}
  \norm{\bb_{21} - \bb_{21}^\rho }_{n,\infty}\le \frac 1{8C_2C_0},
}
where $\bb_{21}^\rho = \bb_{21}* \Phi_\rho $ is the mollification of   $\bb_{21}$ against $\Phi_\rho$.

For any $u \in W^{1,p}_0(\Om)$ with $\norm{u}_{W^{1,p}_0(\Om)}=1$, we decompose
\[
\int_\Om u (\bb_{21} \cdot \nabla v ) \, dx
= \int_\Om u (\bb_{21}- \bb_{21}^\rho)  \cdot \nabla v \, dx
+ \int_\Om ( u \bb_{21}^\rho )\cdot \nabla v \, dx  = : I_1+I_2.
\]
By Lemma \ref{th3-1}, we have
\EQN{
|I_1|&\le \norm{u(\bb_{21}- \bb_{21}^\rho)}_{p} \norm{\nb v}_{p'}
\\
&\le C_0 \norm{\bb_{21}- \bb_{21}^\rho}_{n,\infty}\norm{u}_{W^{1,p}_0} \norm{\nb v}_{p'}
\le \frac 1{8C_2} \norm{\nb v}_{p'}  .
}
By Lemma \ref{app-corr2} with $\delta =1/2$,
\EQN{
|I_2|
&\le   \norm{\bb_{21}}_{n,q} \left(\eta \norm{\nb v}_{p'} +    \frac{C_\eta }{\rho^{3/2}  }\norm{v}_{p'} \right)
}
for any $\eta >0$, where $C_\eta  =C(n,   \Omega, p,   \Phi , \eta)>0$.
Hence, choosing a sufficiently small $\eta$, we get
\EQ{\label{0423c}
C_2 \|\bb_{21} \cdot \nabla v \|_{W^{-1, p'}(\Om)}
\le  \frac14 \|\nabla v
\|_{p'} + \frac {C_3}{\rho^{3/2}  } \|v\|_{p'},
}
where $C_3 =C_3 (n, \Omega,p,q,\Phi,\|\bb_{21}\|_{n,q} )>0$.

From \eqref{0423a}, \eqref{0423b} and \eqref{0423c}, we conclude that
\begin{equation*}
\|v\|_{W^{1,p'}(\Om)} \leq  2C_2 + C_4  \|v\|_{p'}
\end{equation*}
for some $C_4=C \rho^{-3/2}$ with $C>0$ depending on $n, \Om, p,q, \Phi$, and $ \|\bb_{21}\|_{n,q}$.

 On the other hand, from \eqref{eq4.5} and the definition  of $\varepsilon_1$ in \eqref{epsilon-1.def}, it follows that
\[%
\|\bb_2\|_{n, \I, (r  )} \leq 2 \left( \norm{\bb_{21}}_{n,\infty, (r)} + \norm{\bb_{22}}_{n,\infty, (r)} \right) \leq \varepsilon_0 ,
\]%
where $\varepsilon_0$ is the same constant as in Proposition \ref{consequence-key a priori estimate} with $p=2$. Since
\[
v   \in W_0^{1,p'}(\Om ) \hookrightarrow W_0^{1,2}(\Om )  \quad \text{and} \quad  -\De v - \lambda (\bb \cdot \nb v ) +\lambda c v = g \quad \text{in }\Om  ,
\]
we deduce  from  Proposition \ref{consequence-key a priori estimate} that
\begin{equation}\label{v-L2-est}
\|v \|_{W^{1,2}(\Om )} \le C_0 \|g\|_{W^{-1,2}(\Om )} \le  C_1
\end{equation}
for some  $C_1  =C_1 ( n,  \Omega , r, M ,\|{\mathbf b}\|_{2})>0$.
Hence by the interpolation inequality in $L^p$-spaces, have
\[
\begin{split}
\|v\|_{W^{1,p'}(\Om)} & \leq  2C_2 + C_4 \|v\|_{2}^{1-\theta}\|v\|_{(p')^*}^{\theta}  \\
& \leq 2C_2 + C_5\|v\|_{2}  + \frac{1}{2}\|v\|_{W^{1,p'}(\Omega)}
\end{split}
\]
for some $C_5= C_5 (\rho, n,  \Omega,p,q,\Phi,\|\bb_{21}\|_{n,q} )>0$, where  $\theta \in (0,1)$ is defined by
$$
\frac{1}{p'} = \frac{1-\theta}{2} + \frac{\theta}{(p')^*}.
$$
Using  the $L^2$-estimate (\ref{v-L2-est}), we finally  get
$$
\|v\|_{W^{1,p'}(\Om)} \leq 4C_2 + 2 C_1 C_5 ,
$$
which completes the proof of \eqref{0309.v-est}. The whole proof of Proposition \ref{prop4.4} has been completed.
\end{proof}

\begin{proof} [Proof of Theorem \ref{th4-q version}]
Theorem \ref{th4-q version} follows from Proposition \ref{consequence-key a priori estimate} for the case $2 \le p<n$, and from Proposition \ref{prop4.4} for the case
$n' < p<2$.
\end{proof}

\begin{remark}
\begin{enumerate}
\item[\textup{(i)}] Proposition \ref{prop4.4} is more general than the case $n'<p<2$ of Theorem  \ref{th4-q version} because
the condition $\bb_{21} \in L^{n}$ is relaxed to $\bb_{21} \in L^{n,q}$ for some $1\le q<\infty$.
\item[\textup{(ii)}] The proof of Proposition \ref{prop4.4} shows that   $C$ depends on $\bb_{21}$ in a quite explicit way; it is only through $\norm{\bb_{21}}_{n,q}$ and the length scale $\rho$ such that the $\rho$-mollification of $\bb_{21}$, $\bb_{21}*\Phi_\rho$, well approximates $\bb_{21}$ in  $L^{n,\infty}$,  in the sense of \eqref{r.cond2}.
\item[\textup{(iii)}] The convolution kernel $\Phi$ was fixed  arbitrarily. If we choose another kernel, the parameter $\rho$ may change its value.
\end{enumerate}
\end{remark}
\section{Proofs  of Theorems    \ref{th4-q version-strong}  and \ref{th4-q version-strong-u}}
\label{Sec5}

\begin{proof}[Proof  of Theorem    \ref{th4-q version-strong}] Let $\e$ be the smallest number of $1$,  $1/(2C_1C_2)$, and the $\e$ defined in Theorem \ref{th4-q version}  depending only on $n, \Om , p =(q^* )' $, and $M$, where $C_1 = C_1(n,  \Omega, q, M)>0$ and $C_2 = C_2(n,  \Omega,q)>0$ are  the constants to be determined.
We prove Theorem \ref{th4-q version-strong} with this choice of $\e$. Recall that $\bb_2$ satisfies the smallness condition
\begin{equation}\label{smallness 1 for strong}
\norm{\bb_2}_{n,\infty, (r )} + \mathbf{1}_{\{ q^* > 2\}}  \left(\norm{\bb_{22}}_{n,\infty, (r  )} + \norm{\textup{div} \,\bb_{22}}_{n/2,\infty, (r   )}  \right)  \le \e
\end{equation}
for some $r  \in (0,  \diam \Om )$.

By the method of continuity as in the proof of Proposition \ref{consequence-key a priori estimate}, it is sufficient to prove that
there exists  a constant $C>0$ depending only on $n,    \Omega ,q,r$,  $M$, $ \| {\mathbf b}\|_{2}$,   and  $\bb_{21}$ such that
$$
 \| v \|_{W^{2,q} (\Om )} \le C   \|-\De v + \lambda \,  (   \mathbf{b} \cdot \nabla v ) + \lambda cv\|_{q}
$$
for all $ v \in W_0^{1,q}(\Om )\cap W^{2,q}(\Om )$  and $\lambda \in [0,1]$. To this end, let $v   \in W_0^{1,q}(\Om )\cap W^{2,q}(\Om )$  and $\lambda \in [0,1]$ be given, and define $g = -\De v - \lambda ( \mathbf{b} \cdot \nabla v ) + \lambda cv$.
By Lemmas  \ref{th3-2} and \ref{th1-2}, we see that $g \in L^q(\Om)$.
Moreover, if   $p =(q^* )'$, then $n'<p<n$ and $g \in W^{-1, p'}(\Omega )$. Since $\bb_2$ satisfies (\ref{smallness 1 for strong}),   it follows from  Part (ii) of Theorem \ref{th4-q version} with $p =(q^*)'$ that
\begin{equation} \label{v*-0309.est}
\|v\|_{W^{1,q^*}(\Omega)} \leq C_0 \|g\|_{W^{-1, q^*}(\Omega )} \leq C_0 \|g\|_{q},
\end{equation}
where $C_0 =C_0  ( n,    \Omega  ,q,r,  M,  \| {\mathbf b}\|_{2}, \bb_{21} )>0$.
Moreover, since $\Omega$ is a $C^{1,1}$-domain and
\[
-\De v - \lambda ( \mathbf{b}_1 \cdot \nabla v ) + \lambda cv = g + \lambda ( \mathbf{b}_2 \cdot \nabla v ) \quad\mbox{in}\,\,\Om ,
\]
it follows from  \cite[Theorem 2.2]{KiTs},
 Lemma \ref{th3-2}, and (\ref{v*-0309.est}) that
\begin{align*}
 \| v \|_{W^{2,q} (\Om )} & \le C_1 \|g  + \lambda ( \mathbf{b}_2 \cdot \nabla v ) \|_{q}  \\
& \le C_1\|g \|_{q}  +C_1 C_2 \|\bb_2\|_{n, \infty, (r)}\left( \|\nabla^2v\|_{q} + \frac{1}{r} \|\nabla v\|_{q} \right) \\
& \leq  C_1 \left(1 + \frac{C_2 C_3  }{r}\|\bb_2\|_{n, \infty, (r)} \right) \|g \|_{q}  +C_1 C_2   \|\bb_2\|_{n, \infty, (r)}  \|\nabla^2 v\|_{q} ,
\end{align*}
where   $C_1 = C_1 (n,   \Omega,q,   M )>0$, $C_2 = C_2 (n,   \Omega,q)>0$, and  $C_3  = C (n,  \Omega ,q)C_0>0$. By  the choice of $\e$, we see that
$$
\|\bb_2\|_{n,\infty, (r)} \le  \e \le     \frac{1}{2 C_1 C_2}.
$$
 Therefore, we obtain the desired a priori estimate
\[
\| v \|_{W^{2,q} (\Om )} \le 2 C_1 \left(1 + \frac{C_2 C_3}{r} \right) \|g \|_{q}.
\]
The proof of Theorem    \ref{th4-q version-strong} is completed.
\end{proof}

\begin{proof}[Proof of Theorem \ref{th4-q version-strong-u}]  Let $\e$ be the smallest number of $1$,  $1/(2C_1C_2)$, and the $\e$ defined in Theorem \ref{th4-q version} depending only on $n, \Om , p =q^* $, and $M$, where $C_1 = C_1(n,   \Omega,q, M)>0$ and $C_2 = C_2(n,  \Omega,q)>0$ are  the constants to be determined. We prove Theorem \ref{th4-q version-strong-u}  with this choice of $\e$. Recall the smallness condition for $\bb_2$:
\begin{equation*}%
M_r (\bb_2 ) : =\norm{\bb_2}_{n,\infty, (r )} + %
    \norm{\textup{div} \,\bb_{2}}_{n/2,\infty, (r  ) }  \le \e
\end{equation*}
for some $r  \in (0,  \diam \Om )$. Let $u  \in W_0^{1,q}(\Om )\cap W^{2,q}(\Om )$  and $\lambda \in [0,1]$ be given, and define
$f  = -\De u + \lambda \, {\rm div}   ( u   \mathbf{b} ) + \lambda cu    $.
Since $  {\rm div} \, ( u   \mathbf{b} )
 =   \bb \cdot \nabla u   +   ( \text{div} \, \bb    ) u  $, it follows from  Lemmas \ref{th3-2} and   \ref{th1-2}  that $f \in L^q(\Omega) \subset W^{-1,q^*}(\Om )$. Hence by  Part (i) of Theorem \ref{th4-q version} with $p =q^*$, we have
\begin{equation} \label{eqn0108}
 \| u \|_{W^{1,q^*} (\Om )} \le C_0 \|f\|_{W^{-1, q^*}(\Om )} \le C_0  \|f\|_{q},
\end{equation}
where $C_0 =C_0  ( n,   \Omega  ,q,r, M, \| {\mathbf b}\|_{2})>0$. Moreover, since
\begin{equation} \label{eqn0925}
-\De u + \lambda \, {\rm div}  ( u   \mathbf{b}_1 ) + \lambda c u = f - \lambda \left( u \,  {\rm div} \, \bb_2 + \bb_2 \cdot \nabla u \right) \quad\mbox{in}\,\,\Om ,
\end{equation}
it follows from   \cite[Theorem 2.2]{KiTs}, Lemma \ref{th1-2}, Lemma \ref{th3-3}, and (\ref{eqn0108})   that
\begin{align*}
 \| u \|_{W^{2,q} (\Om )} & \le C_1 \|f - \lambda \left( u \,  {\rm div} \, \bb_2 + \bb_2 \cdot \nabla u \right) \|_{q} \\
& \leq  C_1  \|f\|_{q} + C_1 C_2 M_r (\bb_2 ) \left(   \|u\|_{W^{2,q} (\Om )} + \frac{1}{r} \|u\|_{W^{1,q}(\Omega)} \right) \\
&\le C_1 \left(1 + \frac{C_2 C_3}{r} M_r (\bb_2 )\right) \|f\|_{q}+ C_1 C_2 M_r (\bb_2 )\|u\|_{W^{2,q} (\Om )},
\end{align*}
where %
 $C_1 = C_1 (n,   \Omega, q, M ) $, $C_2 = C_2 (n,   \Omega,q)>0$, and    $C_3  = C (n,   \Omega ,q)C_0>0$. By   the choice of $\e$, we have
$$
M_r (\bb_2 ) \le  \e \le  \frac{1}{2 C_1 C_2} .
$$
Therefore,  we obtain
\[
\| u \|_{W^{2,q} (\Om )} \le 2 C_1 \left(1 + \frac{C_2 C_3}{r} \right) \|f \|_{q}.
\]
By the method of continuity, this completes the proof of Theorem \ref{th4-q version-strong-u}.
\end{proof}

\begin{remark}\label{rem51}
The assumption of Theorem   \ref{th4-q version-strong-u} implies that   the decomposition    $\bb_2 =\bb_{21} + \bb_{22}$, where  $\mathbf{b}_{21} \in L^{n}(\Om;\R^n)$ and  $  {\rm div}\, {\mathbf b_{22}} \in L^{n/2,\I} (\Om )$, holds trivially when $\bb_{21} =0 $ and $\bb_{2}=\bb_{22}$.
For the estimate of the right side of \eqref{eqn0925} in $L^q$,
 it is impossible to consider more general $\bb_2$ of the form $\bb_2 =\bb_{21}+\bb_{22}$ with $\bb_{21} \in L^{n}$ having no weak divergence.
\end{remark}

\section{H\"{o}lder regularity for the dual problem} \label{Sec6}

In this section, we prove the H\"{o}lder continuity of  weak solutions $v$ of the dual problem \eqref{bvp-dual}  with $g = \div  \G $ for some  $\G \in L^p(\Omega ; \R^n )$ with $n<p< \infty$. The condition $p>n$ is necessary for a  proof of H\"{o}lder continuity of   $v$ through the Morrey  embedding theorem because    the best we could hope for is that
$\norm{\nb v}_{L^p} \lec \norm{\G}_{L^p}$. Throughout the section, we denote
\[
\Omega_\rho(x_0) = \Omega \cap B_\rho(x_0)\quad\mbox{and} \quad A_k(\rho)   = \{x\in \Om_\rho (x_0):  v(x)  > k \}
\]
for $\rho>0$, $k \in \mathbb{R}$, and $x_0 \in \overline{\Omega}$. Note that $A_k(\rho)$ depends also on $x_0$ and $v$ but we suppress these dependences for the purpose of abbreviation %

We start by proving boundedness of solutions in Subsection \ref{Sec6-1} which relies on a lemma on Caccioppoli type estimates. Then in Subsection \ref{Sec6-3}, we  prove density lemmas and  H\"{o}lder continuity results in the interior and on the boundary assuming that solutions are bounded.

\subsection{Caccioppoli estimates and boundedness of solutions}\label{Sec6-1}

We begin with the following lemma  on Caccioppoli type estimates.

\begin{lemma}
\label{caccioppoli-lemma}  Let $\Omega$ be a bounded Lipschitz domain in $\R^n$ with $n \ge 3$. Then  there exists  a small  number $\varepsilon = \varepsilon (n,\Omega )>0$ such that the following assertion holds.

Assume that   $\bb = \bb_1 + \bb_{2} +\bb_{3}$,   $(\mathbf{b}_1 , \mathbf{b}_2 ) \in L^{n,\I}(\Om;\R^{2n})$, $\mathbf{b}_{3} \in L^{n}(\Om;\R^n)$,   $\div \bb_2 \in L^{n/2,\infty}(\Omega)$, $\div \bb_1  \geq 0$ in $\Omega$,  and
\begin{equation} \label{smallness-0203}
  \left\|   \div   \bb_2   \right\|_{n/2, \infty, (r)} \le \varepsilon  \quad \text{for some } r  \in (0,  \diam \Om ).
\end{equation}
Assume also that $p \in (n,\infty)$, $c \in L^{p^{\#}}(\Omega )$, where   $p^{\#}=np/(n+p)$, and $g = \div  \G $ for some  $\G \in L^p(\Omega ; \R^n )$.

Then there exist constants $C_1 = C_1(n, \Omega, r,  \|\bb  \|_{n,\infty}, \bb_3, \norm{c}_{p^\sharp}) >0$ and $C_2 = C_2(n )>0$ such that if  $v \in W^{1,2}_0(\Omega)$ is a weak solution of \eqref{bvp-dual},  then for every $x_0 \in \overline{\Omega}$,  $0< \tau<\rho  \le  R  \le 2 \, \textup{diam }\Omega$,  and $k  \in \R$, we have
\begin{equation} \label{eqn-0319-23}
\begin{split}
 \int_{A_k (\tau )} |\nabla  v|^2  \,  dx  & \leq  \frac{C_1}{(\rho-\tau)^2}   \int_{A_k  (\rho)} (v - k)^2 \, dx  \\
&   \quad  + C_2 \left(\|\G\|_{L^{p}(\Omega_R (x_0) )}^2 + k^2 \norm{c}_{p^\sharp}^2  \right) |A_k (\rho)|^{1 - \frac 2p}.
\end{split}
\end{equation}
\end{lemma}

\begin{proof}
Let  $x_0 \in \overline{\Omega}$,    $0<\tau < \rho \le  R  \le   2 \,\text{diam } \Omega$, and $k \geq 0$ be fixed. Define $w   = (v-k)^+$, and for a fixed $\rho\in(0,R ]$,  let $\eta \in C_c^\infty(B_\rho(x_0))$ be any  cut-off function with $0 \le \eta \le 1$. Then
using  $w \eta^2 \in W^{1,2}_0(\Omega)$ as a test function for \eqref{bvp-dual}, we obtain
\[
\int_{\Omega}  \nabla v \cdot \nabla (w \eta^2) \, dx  =\int_{\Omega} \left[ w \eta^2 \bb \cdot \nabla v - c v  w \eta^2  - \G \cdot \nabla (w \eta^2)\right] dx.
\]
Since $\nabla v = \nabla w$ and $v = w+k$ on
$A_k(\rho) = \big \{x \in \Omega_\rho(x_0): w \neq 0 \big\}$,
we have
\[
\begin{split}
\int_{\Omega} |\nabla w|^2 \eta^2 \, dx & =-2 \int_{\Omega} \eta w \nabla w \cdot \nabla \eta \, dx -\int_{\Omega} \G \cdot \big( \eta^2 \nabla  w + 2 \eta w \nabla \eta \big) \, dx \\
&\qquad + \int_{\Omega} \big[ w \eta^2 \bb \cdot \nabla w  - c(w+k) w \eta^2\big] \, dx   ,
\end{split}
\]
where  all the integrals can be restricted to $A_k(\rho)$. By Young's inequality,
\EQ{\label{0709a}
\begin{split}
\frac{1}{4} \int_{\Omega} |\nabla ( w\eta)|^2  \, dx
& \leq  C\int_{A_k(\rho)} \Big( w^2 |\nabla \eta|^2 +|\G|^2 \eta^2 \Big) \, dx \\
& \qquad + \int_{\Omega} \big[ w \eta^2 \bb \cdot \nabla w  +   |c|(w+ |k|)  w \eta^2 \big]\,  dx,
\end{split}
}
where $C>0$ is an absolute constant. Now, using the decomposition  $\bb = \bb_1 + \bb_2 +\bb_3$, we write
\[
\int_{\Omega} w \eta^2 \bb \cdot \nabla w  \, dx
=-\int_{\Omega}  w^2 \eta \bb \cdot \nabla \eta \, dx+  \sum_{i=1}^3 \int_{\Omega} w \eta \bb_i \cdot \nabla (w\eta) \,  dx .
\]
Since   $\div   \bb_1  \geq 0$ and $\div   \bb_2  \in L^{n/2,\infty}(\Omega )$,
\EQN{
\sum_{i=1}^2 \int_{\Omega} w \eta \bb_i \cdot \nabla (w\eta) \,  dx
&=
\int_{\Omega}  ( \bb_1 +\bb_2) \cdot \nabla \left(\frac{1}{2}w^2 \eta^2 \right)  dx
\\ &\le
-  \int_{\Omega} \frac{1}{2} (\div \bb_2 )w^2 \eta^2  \, dx .
}
As a consequence, we obtain
\begin{equation} \label{eqn-1-leme61}
\begin{split}
   \frac{1}{4}\int_{\Omega} |\nabla  ( w\eta)|^2  \, dx & \leq C\int_{A_k(\rho)} \Big( w^2 |\nabla \eta|^2 +|\G|^2 \eta^2 \Big) \, dx \\
&    \quad   +   \int_{\Omega} \Big[ - w^2 \eta \bb \cdot \nabla \eta + w  \eta \bb_3 \cdot \nabla (w\eta ) \Big] \,dx \\
&  \quad  +  \int_{\Omega} \left[    \left(\tfrac{1}{2} | \div \bb_2|+|c| \right)  w^2 \eta^2+ |k c|   w \eta^2 \right] \,dx,
\end{split}
\end{equation}
where $C>0$ is an absolute  constant.

Next, we estimate the terms in the second and third integrals on the right hand side of \eqref{eqn-1-leme61}. By H\"{o}lder's inequality, the estimate \eqref{th3-1-eq1} of Lemma \ref{th3-1}, and Young's inequality,
\[
\begin{split}
\left|\int_{\Omega}  w^2 \eta \bb \cdot \nabla \eta \, dx\right| & \leq \| w\eta \bb\|_{2}\|w\nabla \eta\|_2 \\
&  \leq C\|\bb\|_{n, \infty} \|\nabla ( w\eta)\|_2 \|w\nabla \eta\|_2 \\
& \leq \frac{1}{16}  \|\nabla ( w\eta)\|_2^2 + C\|\bb\|_{n, \infty} ^2 \|w\nabla \eta\|_2^2,
\end{split}
\]
where $C = C(n, \Omega)>0$. Using  the estimate (\ref{th3-1-eq3}) in Lemma \ref{th3-1}, we   have
\begin{align} \notag
\left|\int_{\Omega}  w  \eta \bb_3 \cdot \nabla (w\eta ) \, dx\right| & \leq \| w\eta \bb_3 \|_{2}\|\nabla (w\eta )\|_2 \\ \notag
& \leq \left(\frac{1}{32}  \norm{\nabla(w\eta ) }_{2} + C \norm{w\eta }_{2 } \right)\|\nabla (w\eta )\|_2
\\ \label{eq0720a}
& \leq \frac{1}{16}  \|\nabla ( w\eta)\|_2^2 + C\|w  \eta\|_2^2 ,
\end{align}
where $C= C( n,\Omega , \bb_3 )>0$.
Also, if $\tilde{c} = \frac{1}{2}|\div \bb_2 | + |c|$,
then by Lemma \ref{th1-2}  (i), we see that
\[
\begin{split}
 \left|  \int_{\Omega}\tilde{c} w^2 \eta^2 \, dx \right|
& \leq  C \|\nabla ( w\eta)\|_2   \left\|\tilde{c} w \eta \right\|_{W^{-1,2}(\Omega)}    \\
& \leq  C \|\nabla ( w\eta)\|_2  \left\| \tilde{c} \right\|_{n/2, \infty, (r)} \left(\|\nabla (w \eta)\|_2 + \frac{1}{r} \|w \eta\|_2 \right)\\
& \leq C_* \|\nabla ( w\eta)\|_2^2   \Big(\| \tilde{c}\|_{n/2, \infty, (r)}+ \|\tilde{c} \|_{n/2, \infty, (r)}^2  \Big)  + \frac{1}{r^2} \|w \eta\|_2^2
\end{split}
\]
for any $r \in (0, {\rm diam}\, \Omega)$, where $C_*= C (n, \Omega)>0$.
Define
\[
\varepsilon   = \min \left\{ \tfrac{1}{4}, \tfrac{1}{128 C_*} \right\} .
\]
Then since $c \in L^{p^\sharp}(\Omega )$ and $p^\sharp >n/2$, by taking a smaller $r>0$ in \eqref{smallness-0203} if necessary (depending on $\norm{c}_{p^\sharp}$), we have
\[
\| \tilde{c}\|_{n/2, \infty, (r)} \le 4\e,
\]
and
we obtain
\[
\left|  \int_{\Omega} \Big(\tfrac{1}{2} | \div \bb_2 |  + |c|\Big) w^2 \eta^2 \, dx \right|  \leq \frac{1}{16} \|\nabla (w\eta  )\|_2^2 + \frac{1}{r^2} \|w \eta\|_2^2.
\]
Now,   putting the three estimates we just derived into \eqref{eqn-1-leme61}, we have
\begin{equation}\label{eqn-1-leme61-22}
\begin{split}
& \int_{\Omega} |\nabla  ( w\eta)|^2   dx  \leq   \hat{C}  \int_{A_k(\rho)}  \left( |\G|^2 \eta^2 + |k c |w \eta^2\right) dx  \\
& \qquad \qquad  \qquad+  C_0  \left( \|\bb\|_{n,\infty}^2+1\right) \int_{A_k(\rho)}  w^2 \left( |\nabla \eta|^2 +   \eta^2\right) dx ,
\end{split}
\end{equation}
where $\hat{C}>0$ is an absolute constant and $C_0 = C_0(n, \Omega , r, \bb_3, \norm{c}_{p^\sharp} )>0$.

Next, by H\"{o}lder's inequality and Sobolev's inequality,
\[
\int_{A_k(\rho)} |\G|^2 \eta^2 dx \leq \|\G\|_{L^{p}(A_k(\rho))}^2  \|\eta\|_{L^{\frac {2p}{p-2}} (A_k(\rho))}^2 \leq \|\G\|_{L^{p}(\Omega_R (x_0 ) )}^2  |A_k(\rho)|^{1-\frac{2}{p}}
\]
and
\[
\begin{split}
\int_{A_k(\rho)} |k c | w\eta^2 \,  dx & \leq |k| \|c \|_{p^\sharp} \| w \eta\|_{2^*} \|\eta\|_{L^{\frac {2p}{p-2}} (A_k(\rho))}\\
&\leq C |k| \|c \|_{p^\sharp} \|\nabla( w\eta)\|_2 |A_k(\rho)| ^{\frac 12  - \frac 1p}
\\
& \leq \frac{1}{2\hat{C}} \|\nabla( w\eta)\|_2^2 + C k^2 \|c\|_{p^\sharp}^2 |A_k(\rho)|^{1 - \frac 2p},
\end{split}
\]
where $C = C(n )>0$.
Substituting these estimates into  (\ref{eqn-1-leme61-22}), we obtain
\[%
\begin{split}
 \int_{\Omega} |\nabla  ( w\eta)|^2   dx  & \leq C_2 \Big(\|\G\|_{L^{p}(\Omega_R (x_0 )  )}^2 + k^2  \|c\|_{p^\sharp}^2 \Big) |A_k(\rho)|^{1 - \frac 2p} \\
&   \quad +C_1 \int_{A_k(\rho)} w^2 \left( |\nabla \eta|^2 +   \eta^2\right) dx   ,
\end{split}
\]%
where $C_1  = C_1 (n, \Omega , r,  \|\bb\|_{n,\infty} , \bb_3, \norm{c}_{{p^\sharp}})>0$  and  $C_2 = C_2(n ) >0$.
Then the estimate \eqref{eqn-0319-23} immediately follows by taking $\eta \in C_c^\infty(B_\rho(x_0))$ such that
\[
\eta = 1 \,\, \text{on} \,\, B_\tau \quad \text{and} \quad |\nabla \eta| + |\eta|\leq \frac{C(n, \text{diam}\,\Om)}{\rho-\tau}
\]
with $\tau \in (0, \rho)$. This completes the proof of the lemma.
\end{proof}

\begin{remark}
The constants  $C_1$ in Lemma \ref{caccioppoli-lemma} and $C$ in \eqref{eq0720a}  depend on $\bb_3 \in L^n(\Om)$  in the sense of Remark \ref{rem-cons-depend}, i.e., they depend on $\rho>0$ such that $\norm{\bb_3}_{n,\infty,(\rho)}$ is sufficiently small.  Note also that  $n/2< p^\sharp<\min \{ n, p/2 \}$, $p = (p^\sharp )^*$,  and $p^\sharp \to \frac n2+$ as $p \to n+$.
\end{remark}

\begin{remark}\label{rmk6.2}
 If we   need the Caccioppoli estimate (\ref{eqn-0319-23}) only for $k \ge 0$, then   in the proof of Lemma \ref{loc-bdn-lemma}, the last integral
$\int_{\Omega} |c|(w+ |k|) w \eta^2 \,  dx$
in \eqref{0709a} can be replaced by $\int_{\Omega} c^-(w+ |k|) w \eta^2 \,  dx$ and all $|c|$ in the subsequent proof can be replaced by $c^-$. Hence assuming   that  $c^- \in L^{p^\sharp}(\Omega )$ and $c \in L^{n/2,\infty} (\Omega )$, we can prove (\ref{eqn-0319-23}) for all $k \ge 0$, with the constant  $C$ depending on $c$ through $\|c^{-}\|_{L^{p^\sharp}}$.
However, the estimate   (\ref{eqn-0319-23}) with $k \in \R$ will be used later  to prove Theorem \ref{thm-0322}.
\end{remark}

From the Caccioppoli estimate (\ref{eqn-0319-23}), we can deduce  the following result for local and global  $L^\infty$-estimates  for weak solutions of  \eqref{bvp-dual}, by applying  an iteration method  due to De Giorgi.
\begin{lemma}
\label{loc-bdn-lemma}  Under the same assumptions as    Lemma \ref{caccioppoli-lemma}, let  $v \in W^{1,2}_0(\Omega)$ be  a weak solution of \eqref{bvp-dual} with $g = \div  \G $ for some  $\G \in L^p(\Omega ; \R^n )$. Then $v$ is      bounded on $\Omega$. Moreover, for every $x_0 \in \overline{\Omega}$ and $  R  \in (0,  2\,   \textup{diam }\Omega )$, we have
\begin{equation} \label{lemm61-main-est}
\begin{split}
&\sup_{\Omega_{R/2}(x_0)} v^{\pm} \\
 & \quad \leq C\left[ \left(\frac{1}{R^{n}}\int_{\Omega_R(x_0)} |v^{\pm} |^2  \,  dx \right)^{\frac{1}{2}} + R \left(\frac{1}{R^n}\int_{\Omega_{R}(x_0)} |\G|^p \, dx \right)^{\frac{1}{p}}  \right] ,
\end{split}
\end{equation}
where     $C>0$  depends only on $n$,   $\Omega$, $r$, $p$,   $\|\bb  \|_{n,\infty}$, %
$\bb_3$, and $\|c\|_{p^\sharp}$, but is independent of $R$.
\end{lemma}

\begin{proof}    For $s>0$, we write $B_s = B_s (x_0)$ and $\Omega_s = \Omega_s (x_0 )$. Using the same notations as in the proof of Lemma \ref{caccioppoli-lemma}, we choose a cut-off function $\eta \in C_c^\infty(B_R  )$ such that
\begin{equation*} %
\eta = 1 \,\, \text{on} \,\, B_\tau , \quad \eta =0 \,\, \text{on} \,\, B_R \setminus B_{\frac{\tau+\rho}{2}}, \quad \text{and} \quad |\nabla \eta| + |\eta|\leq \frac{C_0}{\rho-\tau} ,
\end{equation*}
where $C_0=C_0 (n,\text{diam}\,\Om)>0$.
Then since
\[
 \| (v-k)^+ \eta\|_2    \leq |A_k(\rho)|^{\frac{1}{n}}\|(v-k)^+\eta\|_{2^*}  \leq C(n)|A_k(\rho)|^{\frac{1}{n}} \|\nabla [ (v-k)^+\eta ]\|_2 ,
\]
it follows from the Caccioppoli estimate \eqref{eqn-0319-23} and Remark \ref{rmk6.2} that  there exists a constant $C>0$ depending only on $n$, $\Omega$, $r$,   $\|\bb  \|_{n,\infty}$, $\bb_3$, and $\|c^- \|_{p^\sharp}$ such that
\begin{equation} \label{G-eqn-11}
\begin{split}
&\int_{A_k(\tau)} (v-k)^2 \, dx \\
&\quad    \leq C \left[ \frac{ |A_k(\rho)|^{\gamma + \frac{2}{p}}}{(\rho-\tau)^2}\int_{A_k(\rho)} (v-k)^2 \, dx
  +  \left( G + |k| \right)^2 |A_k(\rho) |^{1+\gamma} \right] ,
\end{split}
\end{equation}
where  $\gamma =2({1}/{n} - {1}/{p}) >0$ and $G = \|\G\|_{L^{p}(\Omega_R )}$.
Moreover, if $h< k   $, then
\[
\int_{A_k (\rho)} (v-k )^2 \, dx \le \int_{A_k (\rho)} (v-h)^2 \, dx  \le \int_{A_h (\rho)} (v-h)^2 \, dx
\]
and
\[
|A_k (\rho) |   \le \min \left\{ |B_R |, \frac{1}{(k-h)^2} \int_{A_h (\rho)} (v-h)^2 \, dx  \right\}.
\]
Hence  from  \eqref{G-eqn-11}, we easily    deduce that
if $0 \leq h<k$ and  $0< \tau <\rho \leq R$, then
\begin{equation}
\label{20220701}
 \|(v-k)^+\|_{L^2(\Omega_\tau)}  \leq \frac{C}{(k-h)^\gamma}\left(\frac{R^{\frac{n}{p}}}{\rho-\tau} + \frac{|k|+G}{k-h} \right) \|(v-h)^+\|_{L^2(\Omega_\rho)}^{1+\gamma} ,
\end{equation}
where  $C = C(n,   \Omega,  r, \|\bb\|_{n,\infty} ,    %
\bb_3 , \|c^{-}\|_{p^\sharp} )$.

We are now  ready to perform an iteration. Though  the argument is similar to \cite[pp.~70-71]{Han-Lin} and \cite[pp.~221-222]{Giusti}, we give its  details to identify the exponents of $R$ and also for completeness. For $l=0, 1,2, \ldots$, we define
\[
k_l = \left(1- 2^{-l} \right)\kappa \quad\mbox{and}\quad \rho_l = \left(1+2^{-l} \right) \frac{R}{2},
\]
where $\kappa>0$ is to be determined later. Then
taking $k = k_l $, $h = k_{l-1} \in [0, k)$, $\tau =\rho_l$, and $\rho = \rho_{l-1}$ in (\ref{20220701}), we have
\[
  \|(v- k_l )^+\|_{L^2(\Omega_{\rho_l})} \leq  \frac{C 2^{(1+\gamma )l} }{\kappa^\gamma} \left( \frac{1}{R^{1-\frac{n}{p}}} + \frac{\kappa + G}{\kappa} \right) \|(v- k_{l-1} )^+\|_{L^2(\Omega_{\rho_{l-1}})}^{1+\gamma}
\]
for all $l \ge 1$. Assume that $\kappa \ge R^{1-\frac{n}{p}} G$. Then since $p>n$ and $0<R \le 2 \, \text{diam }$,
\[
\frac{1}{R^{1-\frac{n}{p}}} + \frac{\kappa + G}{\kappa} \le
\frac{2}{R^{1-\frac{n}{p}}} +1 \le \frac{C(n,p,\Omega)}{R^{1-\frac{n}{p}}}.
\]
Hence defining
$$
E_l = 2^{(1+\gamma)l /\gamma} \|(v- k_l )^+\|_{L^2(\Omega_{\rho_l})} ,
$$
we derive
\[
   E_l \leq  \frac{ C^{*} }{R^{1-\frac{n}{p}}  \kappa^\gamma}   \left( E_{l-1} \right)^{1+\gamma}
\]
for all $l \ge 1$, where $ C^* =C(n , \Omega , p, \|\bb\|_{n,\infty},
\bb_3,  \|c^{-}\|_{p^\sharp} )$. Define
\[
\kappa = \left( \frac{C^*}{R^{1-\frac{n}{p}}}\right)^{\frac{1}{\gamma}} \|v^+ \|_{L^2 (\Omega_R )} + R^{1-\frac{n}{p}}  G  .
\]
Then since $E_0 \le  \|v^+ \|_{L^2 (\Omega_R )} $,
\[
E_1 \leq  \frac{ C^{*} }{R^{1-\frac{n}{p}} \kappa^\gamma}   \left( E_{0} \right)^{1+\gamma}\le  \frac{C^*}{R^{1-\frac{n}{p}}} \left( \frac{\|v^+ \|_{L^2 (\Omega_R )}}{\kappa}\right)^\gamma E_0 \le E_0 ,
\]
which implies by induction that
\[
E_l \le E_0  \quad\mbox{for all}\,\, l \ge 0 .
\]
Hence for all $l \ge 0$, we have
\[
  \|(v- k_l )^+\|_{L^2(\Omega_{\rho_l})} \le 2^{- (1+\gamma)l /\gamma} E_0 ,
\]
where the right side tends to zero as $l \to\infty$. Therefore, letting $l \to \infty$, we conclude that
\[
\sup_{\Omega_{R/2}} v^+ \le \lim_{l \to \infty} k_l  = \kappa = \left( \frac{C^*}{R^{1-\frac{n}{p}}}\right)^{\frac{1}{\gamma}} \|v^+ \|_{L^2 (\Omega_R )} + R^{1-\frac{n}{p}}  G .
\]
Finally using the definitions  of $G$ and $\gamma$, we see that
\[
 \left( \frac{C^*}{R^{1-\frac{n}{p}}}\right)^{\frac{1}{\gamma}} \|v^+ \|_{L^2 (\Omega_R )} = \left( C^* \right)^{\frac{1}{\gamma}} \left(\frac{1}{R^n} \int_{\Omega_R} |v^+ |^2 \, dx  \right)^{1/2}
\]
and
\[
R^{1-\frac{n}{p}}   G    = R \left(\frac{1}{R^n} \int_{\Omega_R} |\G|^p \, dx  \right)^{1/p}.
\]
This completes the proof of \eqref{lemm61-main-est} for $v^+$. By linearity, the estimate \eqref{lemm61-main-est}  for $v^-$ also follows. Finally, taking $R = 2\,  \text{diam} \, \Omega$ in \eqref{lemm61-main-est}, we obtain
\[
\|v\|_\infty \le C  \left( \|v\|_2 + \|\G\|_p \right).
\]
Therefore,  $v$ is bounded on $\Omega$.
\end{proof}

\begin{remark}  It does not seem to be feasible to implement the Moser iteration to prove Lemma \ref{loc-bdn-lemma} under the smallness assumption \eqref{smallness-0203}. This is because in the Moser  method,  the test function $(v^{+})^{l}\eta^2$ is used and the smallness constant $\varepsilon$ depends on $l$ in each step of the iteration.
\end{remark}
\begin{remark}
Observe that the Caccioppoli estimate (\ref{eqn-0319-23}) is used in the proof of Lemma \ref{loc-bdn-lemma} only for the level constant $k \geq 0$. Hence by Remark \ref{rmk6.2}, the integrability condition $c  \in L^{p^\sharp}(\Omega )$ for $c$ can be relaxed by $c^- \in L^{p^\sharp} (\Omega )$ and $c \in L^{n/2,\infty} (\Omega )$.
\end{remark}

\subsection{H\"{o}lder regularity  of weak solutions}\label{Sec6-3}

\emph{Throughout this subsection, under the same assumptions as Lemma \ref{caccioppoli-lemma}, let  $v \in W^{1,2}_0(\Omega)$ be  a weak solution of \eqref{bvp-dual} with $g = \div  \G $ for some  $\G \in L^p(\Omega ; \R^n )$.}
Then it follows from Lemma  \ref{loc-bdn-lemma} that $v$ is bounded on $\Omega$ and
\begin{equation*}%
\|v\|_\infty \le C  \left( \|v\|_2 + \|\G\|_p \right) .
\end{equation*}
In this subsection, we  show that $v$ is   H\"{o}lder continuous on $\overline{\Omega}$ with some exponent $\bar{\beta} \in (0,    1-n/p ]$:
\[
 v \in C^{\bar{\beta}}(\overline{\Omega}),
\]
by closely following the  De Giorgi iteration method presented in  \cite[Section 7.3]{Giusti}. Let  $x_0 \in \overline{\Omega}$ and  $0< R \le  2\, \textup{diam } \Omega$. For simplicity,  we write
\begin{equation} \label{GMX-def}
\begin{split}
& B_\rho  = B_\rho (x_0 ),  \quad \Omega_r =\Omega_r (x_0 ), \quad A_k (\rho)  = \{x\in\Om_\rho : v(x) >k \}, \\
& G  = \|\G\|_{p} , \quad \overline{M}  =\|v\|_\infty , \quad \chi = G + \overline{M},
\\
& \beta  =1- \frac{n}{p}, \quad\mbox{and} \quad \gamma =\frac{2 \beta}{n} =\frac{2}{n} -\frac{2}{p} .
\end{split}
\end{equation}
 We begin with the following lemma which is an immediate  consequence of Lemma \ref{caccioppoli-lemma}.
\begin{lemma}  For  every   $k \geq -\overline{M}  $ and $0 < \tau<\rho  \le  R$,  we have
\begin{equation} \label{eqn-0319-234}
 \int_{A_k  (\tau )} |\nabla  v|^2  \,  dx  \leq  \frac{C}{(\rho-\tau)^2}   \int_{A_k  (\rho)} (v - k)^2 \, dx
  + C \chi^2 |A_k  (\rho)|^{1-\frac{2}{p}}  ,
\end{equation}
where $C>0$ is a constant  depending only on $n$, $\Omega$, $r$,  $\|\bb  \|_{n,\infty}$, $\bb_3$, and $\|c\|_{p^\sharp }$. Here   $\overline{M}$, $\chi$, and $A_k (\rho)$ are  defined as in \eqref{GMX-def}.
\end{lemma}
\begin{proof}
By Lemma  \ref{caccioppoli-lemma}, we infer that
\[
\begin{split}
 \int_{A_k (\tau )} |\nabla  v|^2  \,  dx  & \leq  \frac{C}{(\rho-\tau)^2}   \int_{A_k  (\rho)} (v - k)^2 \, dx + C \left( G +  |k |   \right)^2 |A_k  (\rho)|^{1-\frac{2}{p}} \\
 \end{split}
\]
for every  $k \in \R$ and $0 < \tau<\rho  \le  R$, where $C = C(n, \Omega, r,  \|\bb  \|_{n,\infty}, \bb_3, \|c\|_{p^\sharp})$ is a positive constant.  On one hand,  when $ -\overline{M}  \le k  \le \overline{M}$, we have
$ G + |k   | \le    \chi$, and therefore  \eqref{eqn-0319-234} is obtained. On the other hand, when $k > \overline{M}$,  \eqref{eqn-0319-234} is trivial as both of its sides are zero.
\end{proof}
Next, we derive that following result which is slightly more general than Lemma \ref{loc-bdn-lemma}.
\begin{lemma} \label{level-set-bounded-sol-lm}  Let  $x_0 \in \overline{\Omega}$ and  $0< R \le    \textup{diam } \Omega$.
 Then  for every  $k_0  \ge -\overline{M}$, we have
\begin{equation} \label{eqn-1204-2}
\begin{split}
\sup_{\Omega_{R/2} } (v-k_0  )  & \leq C \left(\frac{1}{R^n} \int_{A_{k_0}(R)} (v-k_0 )^2 \, dx \right)^{\frac{1}{2}} \left(\frac{|A_{k_0 } (R)|}{R^n} \right)^{\frac{\alpha}{2}} \\
& \quad  + C  \chi     R^{\beta} ,
\end{split}
\end{equation}
where $\alpha$ is  the positive  solution of the equation  $\alpha^2 + \alpha = \gamma$ and   $C>0$ is a constant  depending only on $n$, $\Omega$, $r$,  $\|\bb  \|_{n,\infty}$, $\bb_3$, and $\|c\|_{p^\sharp }$;  recall  that $\Om_\rho$, $\overline{M}$, $A_{k_0}(R)$, $\chi$, $\gamma$, and $ \beta$ are defined  as in \eqref{GMX-def}.
\end{lemma}
\begin{proof} For $k_0  \ge -\overline{M}$, we define
\[
w =v-k_0 \quad\mbox{and}\quad \bar{A}_k (\rho )  =  \{x\in\Om_\rho (x_0): w(x) >k\}  .
\]
Then since $\bar{A}_k (\rho)=A_{k+k_0}(\rho)$, it follows  from (\ref{eqn-0319-234}) (with $k$ replaced by $k+k_0 \ge -\overline{M}$)   that
\[
 \int_{\bar{A}_k (\tau )} |\nabla  w|^2  \,  dx  \leq  \frac{C}{(\rho-\tau)^2}   \int_{\bar{A}_k (\rho)} (w - k)^2 \, dx
  + C \chi^2 |\bar{A}_k (\rho)|^{1-\frac{2}{p}}
\]
for  every  $k \geq   0 $ and $0 < \tau<\rho  \le  R$.  Following the proof of the estimate (\ref{G-eqn-11}), we can deduce that if    $k \geq   0 $ and $0 < \tau<\rho  \le  R$, then
\[
 \int_{\bar{A}_k(\tau)} (w-k)^2 \, dx
    \leq C \left[ \frac{ |\bar{A}_k(\rho)|^{\frac{2}{n}}}{(\rho-\tau)^2}\int_{\bar{A}_k(\rho)} (w-k)^2 \, dx
  +  \chi^2 |\bar{A}_k(\rho) |^{1+\gamma} \right]  ,
\]
which is indeed the key inequality   (7.35) for the proof of  \cite[Proposition 7.1]{Giusti}. Therefore, by exactly the same   argument as in the proof of  \cite[Proposition 7.1]{Giusti}, we can conclude that
\[
\sup_{\Omega_{R/2}(x_0)} w   \leq C \left(\frac{1}{R^n} \int_{\bar{A}_{0}(R)} w^2 \, dx \right)^{\frac{1}{2}} \left(\frac{|\bar{A}_{0}(R)|}{R^n} \right)^{\frac{\alpha}{2}}
  +  C \chi R^{\beta},
\]
which   is nothing but the desired estimate (\ref{eqn-1204-2}).
\end{proof}

For $x_0 \in \overline{\Omega}$ and $R>0$, we write
\begin{equation} \label{osc-def}
\begin{split}
& M_R (x_0, v) =\sup_{\Om_R (x_0 )} v , \quad m_R (x_0, v) = \inf_{\Om_R (x_0 )} v,   \\
& \text{osc}_{x_0} \left( v, R \right)  = M_R (x_0, v) -m_R (x_0, v).
\end{split}
 \end{equation}

\begin{lemma}[Density lemma (interior case)] \label{density-lemma} For  $x_0 \in \Omega$ and $0<   R < \textup{dist}(x_0,\partial \Omega) /2$,  let  $k_0 =  [M_{2R}(x_0, v)+m_{2R}(x_0, v)]/2$. Assume that
\begin{equation}\label{cond-density lemma}
 |A_{k_0}(R)|  \leq \tau_0   |B_R| \quad\mbox{for some}\,\,  \tau_0  \in (0,1).
\end{equation}
Then for a  positive integer $\nu$   satisfying
\begin{equation} \label{nu-R-01-31}
\textup{osc}_{x_0} \left(v,  2R  \right)  \geq 2^{\nu+1}\chi  R^{\beta} ,
\end{equation}
 we have
\[
|A_{k_\nu}(R)| \leq  C_{\tau_0}  \, \nu^{-\frac{n}{2(n-1)}}|B_R|,
\]
where
\begin{equation} \label{k-nu-choice-0131}
k_\nu = M_{2R}(x_0 , v) - 2^{-\nu-1} \, \textup{osc}_{x_0} \left(v,  2R  \right)  .
\end{equation}
Here $B_R$, $A_{k_0}(R)$, $\chi$,   and  $ \beta$ are defined  as in \eqref{GMX-def}, and $ C_{\tau_0}>0$ is a constant  depending on $n$, $\Omega$, $r$,  $\|\bb  \|_{n,\infty}$, $\bb_3$, $\|c\|_{p^\sharp}$,  and $\tau_0$.
\end{lemma}
\begin{proof} For $k_0 \le  h < k$, we  define $w:\R^n \to \R$ by
\[
w  = \left\{
\begin{array}{ll}
k-h & \quad \text{if} \quad v \geq k, \\
v-h & \quad \text{if} \quad h < v < k,\\
0 & \quad \text{if} \quad v \leq h ,
\end{array} \right.
\]
where $v$ is extended  to $\R^n$ by defining zero outside  of $\Omega$. Then $w =  0$  on $B_R(x_0)  \setminus A_{k_0}(R)$ and $|B_R(x_0)  \setminus A_{k_0}(R)| \geq (1-\tau_0)|B_R(x_0)|$. Hence by  the Sobolev-Poincar\'e inequality (see \cite[Theorem 3.16]{Giusti}), we obtain
\begin{equation} \label{poincare-0717}
\left(\int_{B_R} w^{\frac{n}{n-1}} \, dx \right)^{\frac{n-1}{n}} \leq C \int_{B_R} |\nabla w| \, dx = C\int_{A_h(R) \setminus A_k(R)} |\nabla v| \, dx ,
\end{equation}
 where  $C= C(n)(1-\tau_0)^{-\frac{n-1}{n}}$.  %
 Therefore, by the definition of $w$,
\begin{align} \nonumber
(k-h) |A_k(R)|^{\frac{n-1}{n}} & \leq \left(\int_{B_R} w^{\frac{n}{n-1}} \,dx \right)^{\frac{n-1}{n}} \\ \label{density-est-1}
&  \leq C|A_h(R) \setminus A_k(R)|^{\frac{1}{2}} \left(\int_{A_h(R)} |\nabla v|^2 \, dx \right)^{\frac{1}{2}}.
\end{align}
On the other hand,  applying the Caccioppoli estimate  \eqref{eqn-0319-234}  with   $\tau = R$ and $\rho =2R$, we deduce that if $h \ge -\overline{M} $, then
\begin{align*}
 \int_{A_h(R)} |\nabla v|^2 \, dx & \leq \frac{C}{R^2} \int_{A_h(2R)} (v-h)^2 \, dx + C\chi^2 |A_h(2R)|^{1- \frac{2}{p}} \\
 & \leq CR^{n-2}\left( M_{2R} -h\right)^2 + C\chi^2 R^{n -\frac{2n}{p} } \\
& \leq CR^{n-2}\left[\left( M_{2R}   -h\right)^2 + \chi^2 R^{2\beta} \right],
\end{align*}
where $M_{2R}= M_{2R}(x_0, v)$, and $C>0$ is a constant  depending only on $n$, $\Omega$, $r$,  $\|\bb  \|_{n,\infty}$, $\bb_3$, and $\|c\|_{p^\sharp }$.
In addition, by \eqref{nu-R-01-31}, we infer that
\[
M_{2R}  - h \geq M_{2R}  - k_\nu = 2^{-\nu-1} \, \textup{osc} \left(v,  2R  \right) \geq \chi R^\beta \quad \text{if} \,\, h \leq k_\nu.
\]
Hence, it follows that
\[
 \int_{A_h(R)} |\nabla v|^2 dx \leq CR^{n-2}\left( M_{2R}  -h \right)^2  \quad\mbox{for all}\,\,  h \in  [-\overline{M}, k_\nu].
\]
Combining  this estimate and \eqref{density-est-1}, we conclude  that
\begin{equation} \label{density-est-4}
(k-h)^2 |A_k(R)|^{\frac{2(n-1)}{n}} \leq C R^{n-2} |A_h(R) \setminus A_k(R)| \left( M_{2R} - h \right)^2
\end{equation}
for $k_0 \leq h < k \leq k_\nu$.

Now, for each $i =1, 2,\ldots, \nu$, let
$k_i = M_{2R}  - 2^{-i-1}\textup{osc}_{x_0} \left(v,  2R  \right) $.
Then taking $k = k_i$ and $h = k_{i-1}$ in \eqref{density-est-4},  we  obtain
\[
  |A_{k_i}(R)|^{\frac{2(n-1)}{n}} \leq CR^{n-2}|A_{k_{i-1}}(R) \setminus A_{k_{i}}(R)|
\]
for $i =1, 2,\ldots, \nu$. Since $A_{k_i}(R) \subset A_{k_{i-1}}(R) \subset B_R $ for all $i$, we infer that
\[
\begin{split}
\nu |A_{k_\nu}(R)|^{\frac{2(n-1)}{n}} & \leq \sum_{i=1}^\nu |A_{k_i}(R)|^{\frac{2(n-1)}{n}}  \leq CR^{n-2} \sum_{i=1}^\nu |A_{k_{i-1}}(R) \setminus A_{k_{i}}(R)| \\
& \leq C R^{n-2} |A_{k_0}(R)| \leq C  R^{2(n-1)},
\end{split}
\]
and therefore
\[
|A_{k_\nu}(R)| \leq C_{\tau_0} \nu^{-\frac{n}{2(n-1)}}|B_R|,
\]
where $C_{\tau_0}>0$ is a constant  depending on $n$, $\Omega$, $r$,  $\|\bb  \|_{n,\infty}$, $\bb_3$, $\|c\|_{p^\sharp }$, and $\tau_0$.
The assertion of the lemma is proved.
\end{proof}

We  now   prove the interior H\"{o}lder %
estimate.

\begin{lemma}[Interior H\"{o}lder regularity] \label{prop-65} There exists a number  $\beta_1 \in (0, \beta]$  depending only on   $n$, $\Omega$, $p$, $r$,   $\|\bb  \|_{n,\infty}$,     $\bb_3$, and $\|c\|_{p^\sharp }$
 such that for every $x_0 \in \Omega$ and $    \rho  \in (0,   \textup{dist}(x_0, \partial \Omega)/2 )$, we have
\[
\displaystyle{\textup{osc}_{x_0}} \left( v, \rho  \right)   \le C \left( \|v\|_\infty + \|\G\|_{p} \right)  \rho^{\beta_1},
\]
where $C>0$ is a constant  depending only on $n$, $\Omega$, $r$,  $\|\bb  \|_{n,\infty}$, $\bb_3$, and $\|c\|_{p^\sharp }$.
\end{lemma}
\begin{proof}  Let $x_0 \in \Omega$ and $0< R \le \textup{dist}(x_0, \partial \Omega)/2  $ be fixed. As in the proof of Lemma \ref{density-lemma}, we write  $k_0= (M_{2R}+ m_{2R} )/2$, where $M_\rho =M_\rho (x_0, v) $ and $m_\rho =m_\rho (x_0, v) $ are defined in \eqref{osc-def} for $\rho>0$.
We first assume that
\begin{equation}\label{0511a}
|A_{k_0}(R)| \leq \frac{1}{2}|B_R|.
\end{equation}
Then by  Lemma \ref{level-set-bounded-sol-lm}, we have
\begin{align} \notag
M_{R/2}   -k  & \leq C \left( \frac{1}{R^n} \int_{A_{k}(R)} (v-k)^2 \, dx \right)^{1/2} \left(\frac{|A_{k}(R)|}{R^n} \right)^{\frac{\alpha}{2}}  + C\chi R^\beta\\ \label{eqn-0720}
& \leq C_0\left[\left(M_R  -k \right) \left(\frac{|A_{k}(R)|}{|B_R |} \right)^{\frac{1+\alpha}{2}} + \chi R^\beta \right]
\end{align}
for $ k \ge  -\overline{M}$, where  $C_0>0$ is a constant  depending only on $n$, $\Omega$, $r$,  $\|\bb  \|_{n,\infty}$, $\bb_3$, and $\|c\|_{p^\sharp}$.
Let $\nu$ be the smallest positive integer so that
\begin{equation}\label{eqn-0413-00}
C_0  \big(C \nu^{-\frac{n}{2(n-1)}}\big)^{\frac{1+\alpha}{2}}  \leq \frac{1}{2} ,
\end{equation}
where $C = C_{1/2}$ is the constant  defined in Lemma \ref{cond-density lemma} with  $\tau_0=1/2$ which  depends only on $n$, $\Omega$, $r$,  $\|\bb  \|_{n,\infty}$, $\bb_3$, and $\|c\|_{p^\sharp }$. Then using
$$
k = k_{\nu} =M_{2R}   - 2^{-\nu-1} \displaystyle{\textup{osc}_{x_0}} \left( v, 2R \right)  ,
$$
in    \eqref{eqn-0720}, we obtain
\begin{equation}\label{eqn-0413-1}
M_{R/2}   -k_\nu    \leq C_0\left[\left(M_R  -k_\nu \right) \left(\frac{|A_{k_\nu}(R)|}{|B_R |} \right)^{\frac{1+\alpha}{2}} + \chi R^\beta  \right].
\end{equation}

If $\displaystyle{\textup{osc}_{x_0}} \left( v, 2R \right) \geq 2^{\nu+1} \chi R^\beta$, then it follows from Lemma \ref{density-lemma}  with $\tau_0=1/2$, \eqref{eqn-0413-00}, and  \eqref{eqn-0413-1} that
\[
M_{R/2}  - k_\nu \leq \frac{1}{2} \left( M_{2R}  - k_\nu \right) + C\chi R^\beta.
\]
This and the definition of $k_\nu$ in \eqref{k-nu-choice-0131} imply that
\EQN{
\textup{osc}_{x_0} \left( v, R/2 \right) &\le
  (M_{R/2} - k_\nu) + (k_\nu - m_{2R})
\\
& \leq \left(1- \frac{1}{2^{\nu+2}} \right) \displaystyle{\textup{osc}_{x_0}} \left( v, 2R \right) + C\chi R^\beta.
}
On the other hand, if $\displaystyle{\textup{osc}_{x_0}} \left( v, 2R \right)< 2^{\nu+1} \chi R^\beta$, then
\[
\displaystyle{\textup{osc}_{x_0}} \left( v, R/2 \right) \leq \displaystyle{\textup{osc}} \left( v, 2R \right) \leq  2^{\nu+1} \chi R^\beta.
\]
In both   cases, we have
\begin{equation}\label{0511b}
\displaystyle{\textup{osc}_{x_0}} \left( v, R/2 \right)\leq \left(1- \frac{1}{2^{\nu+2}} \right) \displaystyle{\textup{osc}_{x_0}} \left( v, 2R \right) + C2^\nu \chi R^\beta
\end{equation}
under the assumption \eqref{0511a}. If \eqref{0511a} fails to hold, we can repeat the proof for $-v$ which is a solution of \eqref{bvp-dual} with $\G$ replaced by $-\G$, and still get \eqref{0511b}.

Now, by a standard iteration lemma based on (\ref{0511b}) (see \cite[Lemma 7.3]{Giusti}),  we can  choose the number
$$  \beta_1   = \min\big\{\beta, \log_{1/4}(1-2^{-\nu-1})\big\} \in (0, \beta]$$
and obtain
\[
\displaystyle{\textup{osc}_{x_0}} \left( v, \rho \right) \leq C \left[ \left(\frac{\rho}{R_0}\right)^{\beta_1} \displaystyle{\textup{osc}} \left( v, R_0  \right)  + \chi \rho^{\beta_1} \right] \leq C\chi  \rho^{\beta_1}
\]
for all $\rho \in (0, R_0 )$, where $R_0 =  \text{dist}(x_0, \partial \Omega)/2 $. The assertion of the lemma is proved.
\end{proof}

Next, we prove the boundary H\"{o}lder %
estimate. We have the following density lemma on the boundary.

\begin{lemma}[Density lemma (boundary case)] \label{bdr-density-lemma} For $x_0 \in \partial \Omega$ and $0< R <  \textup{diam} \,\Omega  /2$, let $k_0 =  [M_{2R}(x_0, v)+m_{2R}(x_0, v)]/2$. Assume that
\begin{equation}\label{cond-density lemma-bd}
k_0 \geq 0 \quad\mbox{and}\quad  |A_{k_0}(R)|  \leq \tau_0   |B_R|  \quad\mbox{for some}\quad  \tau_0 \in  (0,1).
\end{equation}
Then for  a  positive integer $\nu$   satisfying
\[
\textup{osc}_{x_0} \left(v,  2R  \right)  \geq 2^{\nu+1}\chi  R^{\beta} ,
\]
we have
\[
|A_{k_\nu}(R)| \leq C \nu^{-\frac{n}{2(n-1)}}|B_R|,
\]
where
\[
k_\nu = M_{2R}(x_0 , v)  - 2^{-\nu-1} \, \textup{osc}_{x_0} \left(v,  2R  \right).
\] %
Here $B_R$, $A_k(R)$, $\chi$,  and   $ \beta$ are defined  as in \eqref{GMX-def}, and $C>0$ is a constant  depending only on $n$, $\Omega$, $r$,  $\|\bb  \|_{n,\infty}$, $\bb_3$, and $\|c\|_{p^\sharp }$.
\end{lemma}
\begin{proof} For $k_0 \leq h < k$, let $w$ be defined as in Lemma \ref{density-lemma}. Because $k_0 \geq 0$, we see that $w =  0$  on $B_R(x_0)  \setminus A_{k_0}(R)$. Moreover, we also have
\[ |B_R(x_0)  \setminus A_{k_0}(R)| \geq (1-\tau_0)|B_R(x_0)|.
\]
Therefore, we can apply the Poincar\'e  inequality as in \eqref{poincare-0717}.
From this, the proof of the lemma follows exactly as that of Lemma \ref{density-lemma}.
\end{proof}

\begin{lemma}[Boundary H\"{o}lder regularity] \label{bdr-prop-65} There exists a number  $\beta_2 \in (0, \beta]$ depending only on $n$, $\Omega$, $r$, $p$,   $\|\bb  \|_{n,\infty}$, $\bb_3$, and $\|c\|_{p^\sharp}$ and there exists $R_0 \in (0, \textup{diam}\, \Omega/2)$ depending on $\Omega$ such that for every $x_0 \in \partial \Omega$ and $\rho \in (0, R_0)$, we have
\[
\displaystyle{\textup{osc}_{x_0}} \left( v, \rho  \right)   \le C \left( \|v\|_\infty + \|\G\|_{p} \right)  \rho^{\beta_2},
\]
where $C>0$ is a constant  depending only on $n$, $\Omega$, $r$,  $\|\bb  \|_{n,\infty}$, $\bb_3$, and $\|c\|_{p^\sharp}$.
\end{lemma}
\begin{proof} Since $ \Omega$ is a bounded Lipschitz domain, there are  $R_0 \in (0, \textup{diam}\, \Omega/2)$ and $\th_0 \in (0,1)$ such that
\[
|B_R (x_0) \setminus \Omega_R(x_0)| \geq \th_0 |B_R(x_0)|
\]
for all $x_0 \in \pd \Om$ and $R \in (0, R_0]$. Fix $x_0 \in \pd \Om$ and $R \in (0, R_0/2]$, and let
\[
k_0 = \frac{1}{2} \left[M_{2R}(x_0, v) + m_{2R}(x_0, v)\right] .
\]
We assume without loss of generality that $k_0 \geq 0$, because  otherwise we can just repeat the proof for $-v$ instead. We note that as $A_{k_0}(R) \subset \Omega_R(x_0)$, we have
\[
|A_{k_0}(R)| \leq |\Omega_R(x_0)| \leq \tau_0 |B_R(x_0)| \quad \text{with} \,\, \tau_0 = 1- \th_0 .
\]
Hence, the condition \eqref{cond-density lemma-bd} is satisfied.
Then, as in the proof of Lemma \ref{prop-65}, but applying Lemma \ref{bdr-density-lemma} instead of Lemma \ref{density-lemma}, we  get \eqref{0511b} for all $R<R_0/2$ (with a new  $\nu \in \mathbb{N}$ depending on $\theta_0$, $n$, $\Omega$, $r$, $p$,   $\|\bb  \|_{n,\infty}$, $\bb_3$, and $\|c\|_{p^\sharp}$). Therefore, we can choose
$$\beta_2 = \min\big\{\log_{1/4}(1-2^{-\nu-1}), \beta \big\}$$
so  that
\[
\displaystyle{\textup{osc}_{x_0}} \left( v, \rho \right) \leq C \left[ \left(\frac{\rho}{R_0}\right)^{\beta_2} \displaystyle{\textup{osc}} \left( v, R_0  \right)  + \chi \rho^{\beta_2} \right] \leq C\chi  \rho^{\beta_2}
\]
for all $\rho \in (0, R_0)$.  The proof of the lemma is completed.
\end{proof}

\begin{remark}
 For fixed $x_0$ and $R$, we may change the sign of $v$ in the proof of Lemma \ref{prop-65} to ensure the density condition \eqref{0511a}, and in the proof of Lemma \ref{bdr-prop-65} to ensure $k_0 \ge 0$.  Observe also that we only use the non-negative level constants $k, h, k_0$ in the proofs of Lemmas \ref{bdr-density-lemma} - \ref{bdr-prop-65}. Therefore, as in Remark \ref{rmk6.2},  Lemmas \ref{bdr-density-lemma} - \ref{bdr-prop-65} still hold when we replace the assumption $c \in L^{p^{\sharp}}(\Omega)$ by $c^+ \in L^{n/2, \infty}(\Omega)$ and $c^- \in L^{p^{\sharp}}(\Omega)$.
\end{remark}

We conclude this section with the following theorem which summarizes the results in this section.

\begin{theorem}\label{thm-0322}
Let $\Omega$ be a bounded Lipschitz domain in $\R^n $  with $n \ge 3$. Then there is a small number $\varepsilon =\varepsilon (n, \Omega )>0$ such that the following statement  holds:

Assume that   $\bb = \bb_1 + \bb_{2} +\bb_{3}$,   $(\mathbf{b}_1 , \mathbf{b}_2 ) \in L^{n,\I}(\Om;\R^{2n})$, $\mathbf{b}_{3} \in L^{n}(\Om;\R^n)$,   $\div \bb_2 \in L^{n/2,\infty}(\Omega)$, $\div \bb_1  \geq 0$ in $\Omega$,  and
\[
  \left\|   \div   \bb_2   \right\|_{n/2, \infty, (r)} \le \varepsilon  \quad \text{for some } r  \in (0,  \diam \Om ).
\]
Assume also that $p \in (n,\infty)$, $c \in L^{p^{\#}}(\Omega )$, where   $p^{\#}=np/(n+p)$, and $g = \div  \G $ for some  $\G \in L^p(\Omega ; \R^n )$.

Then if $v \in W^{1,2}_0(\Omega)$ is a weak solution of \eqref{bvp-dual}, then $v$ is H\"{o}lder continuous on $\overline{\Omega}$ with some exponent $\bar{\beta} =\bar{\beta}(n, \Omega, p, r, \|\bb\|_{n,\infty},   \bb_3 , \|c\|_{p^\sharp} ) \in (0,  1- n/p]$
and
\[
\|v\|_{C^{\bar{\beta}}(\overline{\Omega})} \leq C \left( \|v\|_{2} + \|G\|_{p} \right)
\]
for some $C=C(n, \Omega, p, r, \|\bb\|_{n,\infty},   \bb_3 , \|c\|_{p^\sharp} )>0$.
\end{theorem}

\begin{proof}
The theorem follows immediately from Lemmas  \ref{loc-bdn-lemma}, \ref{prop-65}, and \ref{bdr-prop-65}.
\end{proof}

\begin{remark}\label{rem66}
As the constant $\ka$ in the proof of Lemma \ref{loc-bdn-lemma} goes to infinity as $p \to n+$, so is the constant $C$ in \eqref{lemm61-main-est}. Hence our proof won't allow us to take $\bar{\beta}=1- n/p$ no matter how small $p-n$ is.
\end{remark}

\section{Proofs of Theorems \ref{th4} and   \ref{th5}} \label{Sec7}

\begin{proposition} \label{existence-holder-sol} Let $\Omega$ be a bounded $C^1$-domain in $\R^n $ with $n \ge 3$, and let  $M \in (0, \infty )$.
Then there is a small number $\varepsilon =\varepsilon (n, \Omega, M )>0$ such that the following statement  holds:

Assume that   $\bb = \bb_1 + \bb_{2} +\bb_{3}$,   $(\mathbf{b}_1 , \mathbf{b}_2 ) \in L^{n,\I}(\Om;\R^{2n})$, $\mathbf{b}_{3} \in L^{n}(\Om;\R^n)$,   $\div   \bb_2  \in L^{n/2,\infty}(\Omega)$, $\|\bb_1\|_{n, \infty}  \le M$, $\div \, \bb_1  \geq 0$ in $\Omega$, and
\begin{equation} \label{small-b2-02-2023}
  \|\bb_2 \|_{n, \infty, (r)} + \left\|\div   \bb_2   \right\|_{n/2, \infty, (r)}    < \varepsilon \quad \text{for some }   r \in (0,  \textup{diam}\, \Omega ).
\end{equation}
Assume also that $p\in (n, \infty)$, $c \in L^{p^{\sharp}}(\Omega)$, where   $p^{\#}=np/(n+p)$, and $c \geq 0$ in $\Omega$.
 Then for each $g \in W^{-1,p}(\Omega )$, there exists a unique weak solution $v \in W^{1,2}_0(\Omega)$ of \eqref{bvp-dual}. Moreover,  we have
\[
v \in C^{\bar{\beta}}(\overline{\Omega}) \quad\mbox{and}\quad  \|v\|_{C^{\bar{\beta}}(\overline{\Omega})} \leq C\|g\|_{W^{-1,p}(\Omega)}
\]
for some $\bar{\beta} \in (0,  1 -n/p]$, where  $C = C ( n,  \Omega, p, r, M, \|\bb\|_{n,\infty},  \bb_3 , \|c\|_{p^\sharp} )>0$.
\end{proposition}

\begin{proof} Let  $\varepsilon$ be a quarter  of the minimum of the two $\varepsilon$'s  in  Theorem \ref{thm-0322}   and  Theorem \ref{th4-q version} with $p =2$.
Let  $g \in W^{-1,p}(\Omega )$ be given.  By the smallness condition  \eqref{small-b2-02-2023} and absolute continuity of $|\mathbf{b}_{3}|^n $ on $\Om$, there exists  $\rho \in (0, r]$ such that
\[
\|\bb_2 + \bb_3 \|_{n, \infty, (\rho  )} \le 2 \left( \|\bb_2 \|_{n, \infty, (\rho  )}
+ \|\bb_3 \|_{n, \infty, (\rho  )} \right) < 2\eps .
\]
Hence by Theorem \ref{th4-q version} (ii) with   $p= 2$ (and $\bb_{2}+\bb_3$ in place of $\bb_2$),
there exists  a unique weak solution $v \in W^{1,2}_0(\Omega)$ of \eqref{bvp-dual}.  Moreover,
$$\norm{v}_{W^{1,2}_0(\Omega)} \le C \norm{g}_{W^{-1,2}(\Om)} \le C \norm{g}_{W^{-1,p}(\Om)}.
$$
By \cite[Lemma 3.9]{KiTs}, we can choose  $\G \in L^p (\Omega ; \R^n )$ such that
$$
g =\div \G\quad\mbox{in}\,\,\Omega \quad \mbox{and}\quad
\|\G\|_{p} \le C(n, \Omega,p ) \|g\|_{W^{-1,p}(\Omega)}.
$$
Then by Theorem \ref{thm-0322},  we obtain
\[
\|v\|_{C^{\bar{\beta}}(\overline{\Omega})} \leq   C\| v\|_{2}+ C\|\G\|_{p} \le C  \|g\|_{W^{-1,p}(\Omega)}.
\]
The proposition is proved.
\end{proof}

Having proved the H\"{o}lder  regularity of weak solutions of  \eqref{bvp-dual}, we are now ready to  prove Theorems \ref{th4} and   \ref{th5}.
To prove Theorem \ref{th4}, we follow the method in \cite[Theorem 2.3]{KiTs} which makes use of the Calder\'{o}n-Zygmund estimates, the H\"{o}lder continuity of weak solutions of   \eqref{bvp-dual}, and the  Miranda-Nirenberg interpolation theorem (Lemma \ref{MNineq}).  Then Theorem \ref{th5} is deduced from Theorem \ref{th4} by a duality argument. We provide the proofs of both  Theorems \ref{th4} and   \ref{th5} below for completeness.

\begin{proof}[Proof of Theorem \ref{th4}] Let $\varepsilon$ be the smallest number of  the $\varepsilon$ defined in Proposition \ref{existence-holder-sol} and the  $\varepsilon$ defined in Theorem \ref{th4-q version} corresponding to  $p=2s$ (this is   different from $p$), where $s \in (n'/2, n/2)$ is a number to be determined (see \eqref{s-def-2023} below).

Suppose that $g \in W^{-1,2}(\Omega)$. Then by the proof of Proposition \ref{existence-holder-sol}, there exists a unique   weak solution $v \in W^{1,2}_0(\Omega)$ of \eqref{bvp-dual}.

We first  prove Part (i). Suppose that  $g \in W^{-1,p}(\Omega)$. Then since  $p \in (n, \infty)$, it follows from Proposition \ref{existence-holder-sol} that
\begin{equation}\label{holder estimate for v}
v \in C^{\bar{\beta}}(\overline{\Omega}) \quad \text{and} \quad \|v\|_{C^{\bar{\beta}}(\overline{\Omega})} \leq C\|g\|_{W^{-1,p}(\Omega)}
\end{equation}
for some $\bar{\beta} \in (0, 1- n/p ]$, where  $C>0$ depends on  $n,  \Omega, p, r , M,  \|\bb\|_{n,\infty}$,   $\bb_3$, and $ \|c\|_{p^\sharp}$.   Let $v_1 \in W^{1,p}_0(\Omega)$ be a $p$-weak solution of the Dirichlet problem for the Poisson equation:
\begin{equation*}
\left\{
\begin{array}{rr}
-\Delta v_1   =  g    \quad \text{in} \,\, \Omega \quad  \\
v_1   =  0  \quad \text{on} \,\, \partial \Omega,
\end{array} \right.
\end{equation*}
which satisfies
\begin{equation}\label{CZ estimate for v_1}
\|v_1 \|_{W^{1,p}(\Omega )} \le C  \|g\|_{W^{-1,p}(\Omega )}
\end{equation}
(see \cite[Theorem 1.1]{JK} e.g.). Define $v_2 = v -v_1$. Then $v_2 \in W^{1,2}_0(\Omega)$ is a weak solution of
\begin{equation} \label{v2-eqn}
\left\{
\begin{array}{rr}
-\Delta v_2   =   h   \quad \text{in} \,\, \Omega \quad  \\
v_2   =   0  \quad \text{on} \,\, \partial \Omega ,
\end{array} \right.
\end{equation}
where $h  = \bb \cdot \nabla v - c v$.
Now, let $s$ be a fixed number satisfying
\begin{equation} \label{s-def-2023}
\max\left\{\frac{(1-\bar{\beta})n}{2-\bar{\beta}},   1  \right\} < s < \frac{n}{2}.
\end{equation}
Since $g \in W^{-1,p}(\Omega) \subset W^{-1, 2s}(\Omega)$  and $\frac{n}{n-1} < 2s < n$, it follows from Part (ii) of Theorem \ref{th4-q version} that
\begin{equation} \label{eqn-2501-1}
v \in W^{1,2s}_0(\Omega) \quad \text{and} \quad \|v\|_{W^{1,2s}_0(\Omega)} \leq C \|g\|_{W^{-1,p}(\Omega)}.
\end{equation}
As $s <n /2 < p^\sharp =np/(n+p)$, we have
$$
|\bb| \in L^{n,\infty}(\Omega) \subset L^{2s}(\Omega) \quad\mbox{and}\quad c \in L^{p^{\sharp}}(\Omega) \subset L^{s}(\Omega).
$$
By H\"{o}lder's inequality, \eqref{eqn-2501-1},   and (\ref{holder estimate for v}),  we obtain
\[
\|\bb \cdot \nabla v\|_{L^{s}(\Omega)} \leq \|\bb\|_{L^{2s}(\Omega)} \|\nabla v\|_{L^{2s}(\Omega)} \leq C\|\bb\|_{L^{n,\infty}(\Omega)} \|g\|_{W^{-1,p}(\Omega)}
\]
and
\[
\|cv\|_{L^{s}(\Omega)} \leq  \|c\|_{L^{s}(\Omega)} \|v\|_{L^\infty(\Omega)} \leq C \|c\|_{L^{p^{\sharp}}(\Omega)} \|g\|_{W^{-1,p}(\Omega)},
\]
so that
\[
h= \bb \cdot \nabla v - c v \in L^s(\Omega) \quad \text{and} \quad \|h\|_{L^s(\Omega)} \leq C\|g\|_{W^{-1,p}(\Omega)}.
\]
Hence because   $\Omega$ is a bounded $C^{1,1}$-domain, we apply the Calder\'{o}n-Zygmund estimate for the Poisson equation \eqref{v2-eqn} (see \cite[Theorem 9.15]{GilTru} e.g.) to infer that $v_2 \in   W^{2,s}(\Omega)$ and
\[
\|v_2\|_{W^{2,s}(\Omega)} \leq C\|h\|_{L^s(\Omega)} \leq C \|g\|_{W^{-1,p}(\Omega)}.
\]
Moreover, as $v = v_1 + v_2$, it follows from the   Morrey embedding theorem, (\ref{holder estimate for v}), and  (\ref{CZ estimate for v_1})    that
\[
\begin{split}
\|v_2\|_{C^{\bar\beta}(\overline{\Omega})} & \leq \|v\|_ {C^{\bar\beta}(\overline{\Omega})} + \|v_1\|_{C^{\bar{\beta}}(\overline{\Omega})}\\ &\leq \|v\|_ {C^{\bar\beta}(\overline{\Omega})} + C\|v_1\|_{W^{1,p}(\Omega)} \leq C\|g\|_{W^{-1,p}(\Omega)}.
\end{split}
\]
Then letting  $s_1 = \frac{(2-\bar\beta)s}{1-\bar\beta} $ and  applying the Miranda-Nirenberg inequality (Lemma \ref{MNineq}), we infer that $ v_2 \in W^{1,s_1}(\Omega)$ and
\[
\|  v_2\|_{W^{1,s_1}(\Omega)} \leq  C\Big( \|v_2\|_{W^{2,s}(\Omega)} + \|v_2\|_{C^{\bar{\beta}}(\overline{\Omega})} \Big)\leq C\|g\|_{W^{-1,p}(\Omega)}.
\]
Note that $s_1 >n$. Therefore, taking
\[
\delta_1 %
= \min\{p, s_1\} -n  \in (0, p-n] ,
\]
we see that
\[
  v \in W^{1, n+ \delta_1}(\Omega) \quad \text{and} \quad \|  v\|_{W^{1,n+\delta_1}(\Omega)} \leq C\|g\|_{W^{-1,p}(\Omega)}.
\]
The assertion (i) of Theorem \ref{th4} is proved.

We next prove Part  (ii). We only need to consider $g \in L^q(\Omega)$ for $q \in (n/2,\infty )$, sufficiently close to $n/2$.
Suppose that $g\in L^q (\Omega )$ and  $q \in (n/2,p^\sharp  )$.
Then by the Sobolev embedding theorem, we see that $g \in W^{-1, q^*}(\Omega)$ and $q^*= nq/(n-q) \in ( n, p)$. Since $ (q^* )^\sharp =q < p^\sharp$, it follows from Part (i), with $p$ replaced by $q^*$, and the Sobolev embedding theorem   that
\[
\|v\|_{\infty}+\|  v\|_{W^{1,n+\delta_1}(\Omega)} \leq C\|g\|_{W^{-1,q^*}(\Omega)}\leq C\|g\|_{q}
\]
for  some   $\delta_1  \in (0,  q^* -n ]$. Hence, if $q_0$ is chosen so that
\[
  \frac{n}{2} < q_0 <  \frac{n(n+\delta_1)}{2n +\delta_1} \quad\mbox{and}\quad q_0 \le q ,
\]
then %
\[
\| \bb\cdot \nabla v  - c v \|_{q_0} \le \|\bb\|_{\frac{q_0 (n+\delta_1)}{n+\delta_1 -q_0}} \|\nabla v\|_{n+\delta_1} + \|c\|_{q_0} \|v\|_\infty \le C\|g\|_{q} ,
\]
where $C$ depends on $\|\bb\|_{n,\infty}$,  $\|c\|_{p^\sharp}$, and other things. Finally, as  $v \in W_0^{1,n+\delta_1}(\Omega )$  satisfies
\[
-\Delta v = f \quad \text{in} \,\,  \Omega,
\]
 where $f = g + \bb \cdot \nabla v  -cv \in L^{q_0} (\Omega)$, we apply the Calder\'{o}n-Zygmund  regularity estimate to infer that
\[
\|  v\|_{W^{2,q_0}(\Omega)} \leq C\|f\|_{q_0} \leq C\|g\|_{q}.
\]
Taking $\delta_2 = q_0-n/2 \in (0, q-n/2]$, we complete the proof of  Part (ii).
\end{proof}

\medskip

\begin{proof}[Proof of Theorem \ref{th5}]
Recall that for   $s \in (1, \infty)$, we denote by $s'$ its H\"{o}lder conjugate, and by $s^*$ its Sobolev conjugate.

Let $\e>0$ be $\frac{1}{4}$ of  the smallest of the   $\e$'s defined in Theorem \ref{th4-q version}, Theorem \ref{th4}, and Proposition \ref{existence-holder-sol}. Also, let $l_0  = q_0'$ be the H\"{o}lder conjugate of $q_0 =   n/2 +\delta_2$, where $\delta_2  \in (0, 1)$ is the small number defined in Theorem \ref{th4} (ii) corresponding to a fixed $q \in (n/2, p^\sharp )$. We   prove Theorem \ref{th5} with this choice of   $\e$ and $l_0$.
Note  that
$$
n' \le \frac{n}{2}< l_0 ' =q_0 \le q < n \quad\mbox{and}\quad n'< l_0 < \left(\frac{n}{2}\right)' .
$$

We start with the proof of Part (i). Let $g \in C^\infty_c(\Omega)$ be fixed.  Then by  Theorem \ref{th4} (ii), there exists a strong solution $\phi \in W^{1, l_0 '}_0(\Omega) \cap W^{2,l_0 '}(\Omega)$ of the problem
\begin{equation}\label{phi-strong-eqn0127}
\left\{\begin{array}{rr}
 -\Delta \phi -\bb \cdot \nabla \phi  +c\phi = g \quad \text{in }\Om ,\,\,\,\, \\[4pt]
 \phi =0 \quad\text{on } \pd \Om .
 \end{array}
\right.
\end{equation}
Since $\Omega$ is a $C^{1,1}$-domain, there exists a sequence $\{\phi_k\}$  in $C^2 (\Omega ) \cap C^{1,1}(\overline{\Omega})$ such that  $\phi_k =0$ on $\partial \Omega$ and $\phi_k \rightarrow \phi$ in $W^{2,l_0 '}(\Omega)$ as $k \rightarrow \infty$. Due to the hypothesis  \eqref{eq1.3-0}, we have
\begin{equation} \label{u-r.eqn0127}
\int_{\Omega} u \left(-\Delta \phi_k - \bb \cdot \nabla \phi_k +c \phi_k \right) dx =0  \quad \mbox{for all}\,\, k \in \mathbb{N}.
\end{equation}
Since $u \in L^{l_0}(\Omega )$, $c\in L^{l_0 '} (\Omega )$,   $\phi_k \rightarrow \phi$ in $W^{2,l_0 '}(\Omega)$, and   $W^{2,l_0 '}(\Omega) \hookrightarrow L^\infty (\Omega )$, we have%
\[
\lim_{k\rightarrow \infty}\int_{\Omega} \left(- u \Delta \phi_k +c u \phi_k \right)  dx = \int_{\Omega} \left(- u \Delta \phi +c u \phi \right)  dx.
\]
Moreover, by   Lemma \ref{th3-1},
\[
\begin{split}
& \left|\int_{\Omega} u \bb \cdot \nabla \phi_k \, dx - \int_{\Omega} u \bb \cdot \nabla \phi \, dx \right| \\
&\qquad \leq \|u\|_{l_0 } \|\bb \cdot (\nabla \phi_k - \nabla \phi)\|_{l_0 '} \\
& \qquad \leq C \|u\|_{l_0} \|\bb \|_{n,\infty}\|\nabla \phi_k - \nabla \phi\|_{W^{1,l'_0}(\Omega)}  \rightarrow 0 \quad \text{as} \,\, k \rightarrow \infty.
\end{split}
\]
Hence, from \eqref{phi-strong-eqn0127} and \eqref{u-r.eqn0127},  we obtain
\begin{align*}
\int_\Omega ug \, dx &= \int_{\Omega}u \left( - \Delta \phi - \bb \cdot \nabla \phi +c \phi \right) dx  \\
&=\lim_{k\to\infty} \int_{\Omega} u \left(-\Delta \phi_k - \bb \cdot \nabla \phi_k +c \phi_k \right) dx = 0.
\end{align*}
As $g \in C_c^\infty(\Omega)$ is arbitrary, we conclude that $u =0$. The proof of Part (i) is completed.

\medskip

We next prove Part (ii) of Theorem \ref{th5}.  Let $f \in W^{-1, n'-}(\Om)$ be given. Let $m_0=\max \{p',(l_0)^\sharp \} \in (1, n')$, and fix   $m \in (m_0,n')$.
Then as $m  < n' <  2$, it follows from the Sobolev embedding theorem that $W^{1,2}_0 (\Omega) \subset L^{m^*}(\Omega)$ and  $L^{(m^*)'}(\Omega) \subset W^{-1,2}(\Omega)$.
Moreover, as $\mathbf{b}_{3} \in L^{n}(\Om;\R^n)$, there is $  \rho \in (0,  r ]$ such that
\[
\|\bb_2 + \bb_3 \|_{n, \infty, (\rho  )} \le 2 \left( \|\bb_2 \|_{n, \infty, (\rho  )}
+ \|\bb_3 \|_{n, \infty, (\rho  )} \right) < 2\eps .
\]
Hence by Theorem \ref{th4-q version} (ii) (with $p= 2$ and $\bb_{2}+\bb_3$ in place of $\bb_2$),  for each $g \in L^{(m^*)'} (\Omega)$, the dual problem \eqref{bvp-dual} has a unique weak solution $v=Lg \in W^{1,2}_0(\Omega)$  and
\[
\|Lg\|_{W^{1,2}(\Omega)} \leq C \|g\|_{W^{-1,2}_0(\Omega)} \leq C\|g\|_{L^{(m^*)'}(\Omega)}.
\]
Furthermore,  since  $(m^*)'> n/2$, we can apply Theorem \ref{th4} (ii) to conclude that $Lg \in W^{1,s}_0(\Omega) \cap W^{2,s}(\Omega)$ for some $s \in (n/2 ,(m^*)']$. From   the Sobolev embedding theorem, we then deduce  that $L g \in W^{1,s^*}_0(\Omega)$. On the other hand, since $ (s^* ) ' < [(n/2)^* ]'= n'$ and  $f \in W^{-1, n'-}(\Omega )$, it follows that  $f \in
 W^{-1, (s^*)'}(\Omega)$. Hence the map  $g \mapsto \bka{f, Lg}$ is a bounded linear functional on $L^{(m^*)'}(\Omega)$.\footnote{By Remark \ref{rem66}, we only have  $s< (m^*)'$ and so $(s^* )'   >m$. This is why we need to assume higher regularity of $f$ than $W^{-1,m}(\Omega )$ for boundedness on $L^{(m^*)'}(\Omega )$ of the  map $g \mapsto \bka{f, Lg}$.}
  Therefore, by the Riesz representation theorem,  there exists a unique $u \in L^{m^*}(\Omega)$ satisfying
\[
\int_{\Omega} u  g \,  dx = \bka{f, Lg}  \quad \mbox{for all}\,\, g \in L^{(m^*)'}(\Omega).
\]
For any $\phi \in C^2 (\Omega ) \cap C^{1,1}(\overline{\Omega})$ with $\phi_{|\partial \Omega} =0$, we take  $g =-\Delta \phi - \bb \cdot \nabla \phi +c\phi$. Then since   $  (m^*)' <n$ and $   (m^*)' =(m')^\sharp \le p^\sharp$, it follows that $g \in L^{(m^*)'}(\Omega)$ and  $\phi = Lg$.  Hence for any $\phi \in C^2 (\Omega ) \cap C^{1,1}(\overline{\Omega})$ with $\phi_{|\partial \Omega} =0$, we see that
\[
\int_{\Omega} u \left(-\Delta \phi - \bb \cdot \nabla \phi +c\phi \right) dx = \bka{f, \phi}.
\]
This implies that $u$ is a very weak solution of \eqref{bvp} in $L^{m^*}(\Omega)$, which is unique by Part (i)  as $m^* \ge l_0$.

To prove higher regularity of  $u$, we observe that
\[
- \int_{\Omega} u  \Delta \phi \, dx = \bka{h, \phi}
\]
for any $\phi \in C^2 (\Omega ) \cap C^{1,1}(\overline{\Omega})$ with $\phi_{|\partial \Omega} =0$, where
$$
h = f -\div(u \bb)-cu .
$$
Since $1<m<n'$ and $1/m^* +1/p^\sharp + 1/(m')^* <1$, it follows from  the H\"{o}lder inequality in Lorentz spaces (Lemma \ref{Lorentz-Holder}) that $h \in W^{-1, m,\infty}(\Omega)$; indeed, for any $\phi \in W_0^{1,m',1}(\Omega)$,
\[
\begin{split}
& \Big| \bka{f,\phi} + \int_{\Omega} (u\bb\cdot \nabla \phi -uc \phi)\, dx \Big| \\
& \leq \|f\|_{W^{-1,m}(\Omega)} \|\phi\|_{W_0^{1,m'}(\Omega)} + \|u\|_{m^*} \|\bb\|_{n,\infty} \|\nabla \phi \|_{m', 1} \\
&\quad + \|u\|_{m^*} \| c\|_{p^\sharp} \| \phi \|_{(m')^*}
 \leq C \|\phi\|_{W_0^{1,m',1}(\Omega)}.
\end{split}
\]
By the Calder\'{o}n-Zygmund regularity estimate (see \cite[Proposition 3.12]{KiTs}), there exists  a unique weak solution $\overline{u} \in W^{1,m,\infty}_0(\Omega)$ of the Poisson equation
\[
-\Delta \overline{u}  = h \quad \text{in} \,\, \Om
\]
with the homogeneous boundary condition. Note that both $u$ and $\overline{u}$ belong to $L^{m_0}(\Omega)$. Hence
$w = u -\overline{u}$   is a very weak solution in $L^{m_0}(\Omega)$ of the Laplace equation with trivial data. Therefore, by a standard uniqueness result, we infer that   $w = 0$ identically on $\Omega$ and  $u = \overline{u} \in W^{1,m,\infty}_0(\Omega)$. Because $m$ can be arbitrarily close to $n'$, we conclude that $u \in W_0^{1, n'-}(\Omega )$. This completes the proof.
\end{proof}

\section*{Acknowledgments}
We warmly thank Hongjie Dong for fruitful discussions, in particular in producing a counterexample to \cite[Problem 9.6]{GilTru}.
The research of Kim was partially supported by  Basic Science
Research Program through the National Research Foundation of Korea (NRF) funded by the Ministry of Education [No.~NRF-2016R1D1A1B02015245].
The research of Phan was partially supported by Simons Foundation, grant \# 2769369.
The research of Tsai was partially supported by Natural Sciences and Engineering Research Council of Canada (NSERC) grants RGPIN-2018-04137 and RGPIN-2023-04534.

\addcontentsline{toc}{section}{\protect\numberline{\hspace{2mm}}{References}}

\end{document}